\documentclass[a4paper,10pt]{article}

\usepackage{amsmath}  
\usepackage{amssymb}            
\usepackage{amsxtra}
\usepackage{amscd}
\usepackage[mathscr]{euscript}

\everymath{\displaystyle} 

\newtheorem{thm}{Theorem}[subsection]
\newtheorem{prop}[thm]{Proposition}

\newtheorem{rem}[thm]{Remark}

\newcommand{\npmod}[1]{\!\!\pmod{#1}}

\newenvironment{proof}{\par\noindent{\bf[Proof]}}%
                      {$\blacksquare$\noindent\par\vspace{0.5\baselineskip}}
                      {$\blacksquare$\par\noindent}

\font\b=cmr10 scaled \magstep4

\def\bigzerou{\smash{\lower0.7ex\hbox{\b 0}}}
\def\bigastl{\smash{\lower0.7ex\hbox{\b *}}}
\def\bigastu{\smash{\lower2.7ex\hbox{\b *}}}

\def\addots{\mathinner
    {\mkern1mu\raise1pt\hbox{.}\mkern2mu
    \raise4pt\hbox{.}\mkern2mu\raise7pt\vbox{\kern7pt\hbox{.}}\mkern1mu}}

\makeatletter
\renewcommand{\subsection}{\@startsection%
  {subsection}%
  {2}%
  {0mm}%
  {\baselineskip}%
  {-0.2\parindent}%
  {\normalfont\normalsize\upshape\bfseries}}%
\def\@dotsep{1.5}
\def\@pnumwidth{1em}
\makeatother

\title{On supercuspidal representations of $SL_n(F)$ associated with
  tamely ramified extensions} 
\author{Koichi Takase
        \thanks{The author is partially supported by 
                    JSPS KAKENHI Grant Number JP 16K05053}}
\date{}

\begin{document}   


%
%

\maketitle

\begin{abstract}
We will give an explicit construction of irreducible suparcuspidal
representations of the special linear group over a non-archimedean
local field  and will speculate its Langlands
parameter by means of verifying the Hiraga-Ichino-Ikeda formula of
the formal degree of the supercuspidal representations.

\end{abstract}

\section{Introduction}
\label{sec:introduction}

Although the $L$-packets of the
irreducible supercuspidal representations of the special linear group
over a non-archimedean local field  are well-understood
by the works of \cite{LabesseLanglands1979}, \cite{GelbartKnapp1982}, 
\cite{MoySally1984} or of \cite{HiragaSaito2012}, 
little is known for the explicit determination of
the Langlands parameter of an individual explicitly constructed
supercuspidal representation. 

In this paper, we will construct quite
explicitely some supercuspidal representations, 
associated with tamely ramified extensions, 
of the special linear
group over a non-dyadic non-archimedean local field and will speculate its 
Langlands parameter by showing the formula of the formal degree
established by Hiraga, Ichino and Ikeda \cite{HiragaSaito2012} 
(see the subsection \ref{subsec:concluding-remark} for the
conclusions). 

The main results of this paper are Theorem 
\ref{th:construction-of-supercuspidal-representation} (a 
construction of supercuspidal representations associated with a tamely
ramified extension of the base field), 
Theorem 
\ref{th:formal-degree-wrt-euler-poincare-measur} (an explicit formula
of the formal degree of the supercuspidal representation) 
and Theorem 
\ref{th:ratio-of-gamma-factor} (showing Hiraga-Ichino-Ikeda formula of
the formal degree in the form of Gross and Reeder
\cite{GrossReeder2010}).

The first theorem is proved in Section 
\ref{sec:construction-of-supercuspidal-representation}. Basic ideas
and arguments are these of Shintani \cite{Shintani1968} with a small
modification to our case of the special linear group (the original
Shintani's paper treats the subgroup of the general linear group of
unit determinant). 

The second theorem is proved in Section 
\ref{sec:formal-degree-of-supercuspidal-representation}. 
Our argument is based upon a general theory,
developed by \cite{Takase2019}, of explicit description of irreducible
representations of hyperspecial open compact subgroups associated with
regular adjoint orbits. 

The third theorem is proved in Section 
\ref{sec:induced-representation-of-weil-group} under the assumption
that the tamely ramified extension is Galois extension. 

\section{Results}
\label{sec:results}

\subsection[]{}
Let $F$ be a non-dyadic non-archimedean local field. The integer ring
of $F$ is 
denoted by $O_F$ with the maximal ideal $\frak{p}=\frak{p}_F$
generated by $\varpi=\varpi_F$. The residue class field 
$\Bbb F=O_F/\frak{p}$ is a finite field of $q$ elements. Fix a
continuous unitary additive character $\tau:F\to\Bbb C^{\times}$ such
that
$$
 \{x\in F\mid\tau(xO_F)=1\}=O_F.
$$
Then 
$\widehat\tau(\overline x)=\tau\left(\varpi^{-1}x\right)$ 
($x\in O_F$) gives a non-trivial unitary additive character 
$\widehat\tau:\Bbb F\to\Bbb C^{\times}$.

The special linear group $G=SL_n$ of degree $n$ is a smooth connected
semisimple group scheme over $O_F$ whose Lie algebra is denoted by
$\frak{g}=\frak{sl}_n$. For any commutative $O_F$-algebra $A$, the
group and the $A$-Lie algebra of the $A$-valued points are
$$
 G(A)=\{g\in M_n(A)\mid \det g=1\}
$$
and
$$
 \frak{g}(A)=\{X\in M_n(A)\mid\text{\rm tr}\,X=0\}
$$
where $M_n(A)$ if the $A$-algebra of the square matrices of size $n$
with entries in $A$. 

Throughout this paper, except the subsection 
\ref{subsec:l-factor-of-induced-rep-of-weil-group}, we will
assume that the characteristic $p$ of $\Bbb F$ does not divide $n$ so
that the trace form
$$
 \frak{g}(\Bbb F)\times\frak{g}(\Bbb F)\to\Bbb F
 \qquad
 \left((X,Y)\mapsto\text{\rm tr}(XY)\right)
$$
is non-degenerate. 

For any integers $0<l<r$, the canonical group homomorphism 
$$
 G(O_F/\frak{p}^r)\to G(O/\frak{p}^l)
$$
is surjective, and its kernel
is denoted by $G(\frak{p}^l/\frak{p}^r)$. If 
$r=l+l^{\prime}$ with $0<l^{\prime}\leq l$, then we have a group
isomorphism
$$
 \frak{g}(O/\frak{p}^{l^{\prime}})\,\tilde{\to}\,
 G(\frak{p}^l/\frak{p}^r)
 \quad
 \left(\overline X\mapsto\overline{1_n+\varpi^lX}\right).
$$
Fix an integer $r\geq 2$ and put 
$r=l+l^{\prime}$ with the minimum integer $l$ such that 
$0<l^{\prime}\leq l$, that is 
$$
 l^{\prime}=\left[\frac r2\right]
           =\begin{cases}
             l&:\text{\rm if $r=2l$},\\
             l-1&:\text{\rm if $r=2l-1$}.
            \end{cases}
$$
Let $K/F$ be a field extension of degree $n$. 
The ramification index and the
inertial degree of $K/F$ are denoted by
$$
 \text{\rm $e=e(K/F)$ and $f=f(K/F)$}
$$
respectively. Since $n$ is prime to $p$, the extension $K/F$ is tamely
ramified 
and there exists a prime element $\varpi_K\in O_K$ such that 
$\varpi_K^e\in O_{K_0}$ where $K_0$ is the maximal unramified
subsextionsion of $K/F$. 

We will identify $K$ with a
$F$-subalgebra of the matrix algebra $M_n(F)$ 
by means of the regular representation with respect to
an $O_F$-basis of $O_K$. Take a generator $\beta$ of $O_K$ as an
$O_F$-algebra. Then Shintani \cite{Shintani1968} shows that the
modulo $\frak{p}$ reduction of the 
characteristic polynomial $\chi_{\beta}(t)\in O_F[t]$ of 
$\beta\in M_n(O_F)$ gives the minimal polynomial of 
$\beta\npmod{\frak p}\in M_n(\Bbb F)$, that is 
$\beta\in\frak{gl}_n(O_F)$ is smoothly regular with respect to 
$GL_n$ in the terminology of \cite{Takase2019}. We can assume that 
$T_{K/F}(\beta)=\text{\rm tr}\,\beta=0$ so that 
$\beta\in\frak{g}(O_F)$ is smoothly regular with respect to 
$G=SL_n$. Define a character
$$
 \psi_{\beta}:G(\frak{p}^l/\frak{p}^r)\to\Bbb C^{\times}
$$
by 
$\psi_{\beta}(h)
 =\tau\left(\varpi^{-l^{\prime}}\text{\rm tr}(X\beta)\right)$ for all 
$h=\overline{1_n+\varpi^lX}\in G(\frak{p}^l/\frak{p}^r)$. 
Let $\delta$ be an irreducible representation of $G(O_F/\frak{p}^r)$
such that $\delta|_{G(\frak{p}^l/\frak{p}^r)}$ contains the character
$\psi_{\beta}$. We will regard $\delta$ as a representation of
$G(O_F)$ via the canonical surjection $G(O_F)\to
G(O_F/\frak{p}^r)$. Then our first result is

\begin{thm}\label{th:construction-of-supercuspidal-representation}
If $r\geq 2e$, then the compactly induced representation 
$\text{\rm ind}_{G(O_F)}^{G(F)}\delta$ is an irreducible supercuspidal
representation of $G(F)$.
\end{thm}

\subsection[]{}
The general theory developed in 
\cite{Takase2019} says that the irreducible representation 
$\delta$ of $G(O_F/\frak{p}^r)$ as above is parametrized by 
the character 
$$
 \theta:G_{\beta}(O_F/\frak{p}^r)\to\Bbb C^{\times}
$$
such that $\theta=\psi_{\beta}$ on 
$G_{\beta}(O_F/\frak{p}^r)\cap G(\frak{p}^l/\frak{p}^r)$. The explicit
realization of $\delta$, which is recalled in the subsection 
\ref{subsec:explicit-description-of-delta}, gives 

\begin{prop}\label{prop:dimension-of-delta}
$$
 \dim\delta
 =\frac{q^{rn(n-1)/2}}
       {(O_F^{\times}:N_{K/F}(O_K^{\times}))}\cdot
  \frac{\prod_{k=1}^n(1-q^{-k})}
       {1-q^{-f}}.
$$
\end{prop}

Let $d_{G(F)}(x)$ be the Haar measure on $G(F)$ with respect to which
the volume of $G(O_F)$ is $1$. Then the Euler-Poincar\'e measure 
$\mu_{G(F)}$ on $G(F)$ is
$$
 d\mu_{G(F)}(x)
 =(-1)^{n-1}q^{n(n-1)/2}\prod_{k=1}^{n-1}(1-q^{-k})\cdot 
  d_{G(F)}(x)
$$
(see \cite[3.4, Theor\'eme 7]{Serre1971}).
Since the formal degree of the supercuspidal representation 
$\text{\rm ind}_{G(O_F)}^{G(F)}\delta$ (assuming $r\geq 2e$) 
with respect to $d_{G(F)}(x)$
is the dimension of $\delta$, 
Proposition \ref{prop:dimension-of-delta} gives our second result 

\begin{thm}\label{th:formal-degree-wrt-euler-poincare-measur}
Assume $r\geq 2e$. Then the formal degree of the supercuspidal
representation $\text{\rm ind}_{G(O_F)}^{G(F)}\delta$ with respect to
the Euler-Poincar\'e measure on $G(F)$ is
$$
 \frac{q^{(r-1)n(n-1)/2}}
      {(O_F^{\times}:N_{K/F}(O_K^{\times}))}\cdot
 \frac{1-q^{-n}}
      {1-q^{-f}}.
$$
\end{thm}

\subsection[]{}
\label{subsec:ratio-of-gamma-factors}
Now suppose that the field extension $K/F$ is Galois. Let 
$$
 \delta_K:K^{\times}\,\tilde{\to}\,
     W_K^{\text{\rm ab}}=W_K/\overline{[W_K,W_K]}
$$
be the isomorphism of the local class field theory with the Weil group
$W_K$ of $K$. We assume that 
$$
 \delta_K(\varpi_K)\in W_K^{\text{\rm ab}}
  \subset\text{\rm Gal}(K^{\text{\rm ab}}/K)
$$
induces the geometric Frobenius automorphism of $K^{\text{\rm ur}}/K$
where $K^{\text{\rm ab}}$ and $K^{\text{\rm ur}}$ are the maximal
abelian and the maximal unramified extension of $K$ respectively. 
Then the relative Weil group 
$$
 W_{K/F}=W_F/\overline{[W_K,W_K]}
 \subset\text{\rm Gal}(K^{\text{\rm ab}}/F)
$$
sits in a group extension 
\begin{equation}
 1\to K^{\times}
  \xrightarrow{\delta_K} W_{K/F}
  \xrightarrow{\text{\rm res.}}\text{\rm Gal}(K/F)\to 1
\label{eq:relative-weil-group-is-group-extension-of-fundamental-class}
\end{equation}
corresponding to the fundamental class 
$[\alpha_{K/F}]\in H^2(\text{\rm Gal}(K/F),K^{\times})$ of the local
class field theory. 

Let $\theta:G_{\beta}(O_F/\frak{p}^r)\to\Bbb C^{\times}$ be the
character which parametrizes the irreducible representation $\delta$ of
$G(O_F/\frak{p}^r)$. Since we have
$$
 G_{\beta}(O_F/\frak{p}^r)
 =\left\{\overline x\in\left(O_K/\frak{p}_K^{er}\right)^{\times}
   \biggm| N_{K/F}(x)\equiv 1\npmod{\frak{p}^r}\right\},
$$
take an extension of $\theta$ to a character of 
$\left(O_K/\frak{p}_K^{er}\right)^{\times}$ and consider it as a
  character of $O_K^{\times}$ via the canonical surjection 
$O_K^{\times}\to\left(O_K/\frak{p}_K^{er}\right)^{\times}$ and extend
it to a character of $K^{\times}$, which is denoted also by
$\theta$, by fixing any value $\theta(\varpi_K)\in\Bbb C^{\times}$. 
Then the group homomorphism 
$$
 \Theta:W_{F/K}
    \xrightarrow{\text{\rm Ind}_{K^{\times}}^{W_{K/F}}\theta}
       GL_{\Bbb C}(V)
     \xrightarrow{\text{\rm canonical}}PGL_{\Bbb C}(V)
$$
($V$ is the representation space of the induced representation 
$\text{\rm Ind}_{K^{\times}}^{W_{K/F}}\theta$) is independent, 
up to the conjugate in $PGL_{\Bbb C}(V)$,  of the
     choice of the extension $\theta$ (see Proposition 
\ref{prop:rep-theta-of-relative-weil-group-is-indep-of-extension}). 
Note that 
$$
 \dim_{\Bbb C}V=(W_{K/F}:K^{\times})=n
$$
and that $PGL_{\Bbb C}(V)$ is the dual group of $G=SL_n$. Put
$$
 \rho:W_F\xrightarrow{\text{\rm canonical}}W_{K/F}
   \xrightarrow{\Theta} PGL_{\Bbb C}(V)
$$
and define a representation of the Weil-Deligne group
\begin{equation}
 \varphi:W_F\times SL_2(\Bbb C)
          \xrightarrow{\text{\rm projection}}W_F
          \xrightarrow{\rho}PGL_{\Bbb C}(V).
\label{eq:representation-of-weil-deligne-group-for-supercuspidal-rep}
\end{equation}
Let us denote by $A_{\varphi}$ the centralizer of the image of $\varphi$ in 
$PGL_{\Bbb C}(V)$. Then we will show that
\begin{equation}
 |A_{\varphi}|=(O_F^{\times}:N_{K/F}(O_K^{\times}))\cdot f
\label{eq:order-of-centralizer-of-image-of-varphi-in-pgl(v)}
\end{equation}
(see Proposition
\ref{prop:a-theta-is-character-group-of-gal-k-over-f}). 
Our third result is

\begin{thm}\label{th:ratio-of-gamma-factor}
$$
 \frac 1{|A_{\varphi}|}\left|\frac{\gamma(\varphi)}
                                 {\gamma(\varphi_0)}\right|
 =\frac{q^{(r-1)n(n-1)/2}}
       {(O_F^{\times}:N_{K/F}(O_K^{\times}))}\cdot
  \frac{1-q^{-n}}
       {1-q^{-f}}.
$$
\end{thm}

Here $\gamma(\varphi)=\gamma(\varphi,\text{\rm Ad},0)$ is the special
value of the gamma factor
$$
 \gamma(\varphi,\text{\rm Ad},s)
 =\varepsilon(\varphi,\text{\rm Ad},s)\cdot
  \frac{L(\varphi,\text{\rm Ad},1-s)}
       {L(\varphi,\text{\rm Ad},s)}
$$
associated with the representation 
$$
 \text{\rm Ad}\circ\varphi:W_F\times SL_2(\Bbb C)
  \xrightarrow{\varphi}PGL_{\Bbb C}(V)
  \xrightarrow{\text{\rm Ad}}
  GL_{\Bbb C}(\widehat{\frak g})
$$
of the Weil-Deligne group, where 
$\widehat{\frak g}=\frak{pgl}_{\Bbb C}(V)=\frak{sl}_{\Bbb C}(V)$ is
the Lie algebra of $PGL_{\Bbb C}(V)$. Another representation 
$$
 \varphi_0:W_F\times SL_2(\Bbb C)
  \xrightarrow{\text{\rm proj.}}SL_2(\Bbb C)
  \xrightarrow{\text{\rm Sym}_{n-1}}GL_n(\Bbb C)
  \xrightarrow{\text{\rm can.}}PGL_n(\Bbb C)
$$
of the Weil-Deligne group, with the symmetric tensor representation 
$\text{\rm Sym}_{n-1}$, is the {\it principal parameter} 
(that is, corresponding to the Steinberg representation, 
see the subsection \ref{subsec:principal-parameter-of-weil-group}). 

\subsection[]{}
\label{subsec:concluding-remark}
Since $G=SL_n$ is split over $F$, the $L$-group of $G$ is 
$PGL_n(\Bbb C)$. 
Theorem \ref{th:formal-degree-wrt-euler-poincare-measur} and 
Theorem \ref{th:ratio-of-gamma-factor} shows that the formula due to 
Hiraga-Ichino-Ikeda \cite{H-I-I2008} of the formal degree of 
the suparcuspidal
representation $\text{\rm ind}_{G(O_F)}^{G(F)}\delta$ is valid with
the representation of Weil-Deligne group given by 
\eqref{eq:representation-of-weil-deligne-group-for-supercuspidal-rep}. So
we can speculate that the representation 
\eqref{eq:representation-of-weil-deligne-group-for-supercuspidal-rep}
is the Langlands (or Arthur)  parameter of the supercuspidal
representation $\text{\rm ind}_{G(O_F)}^{G(F)}\delta$.

\section{Construction of supercuspidal representation}
\label{sec:construction-of-supercuspidal-representation}

\subsection[]{}
In this section we will prove Theorem
\ref{th:construction-of-supercuspidal-representation}. It is
sufficient to show the following two propositions on the compactly
induced representation $\text{\rm ind}_{G(O_F)}^{G(F)}\delta$;

\begin{prop}\label{prop:admissibility-of-compactly-induced-rep}
If $K/F$ is unramified or $r\geq 4$, then 
$\text{\rm ind}_{G(O_F)}^{G(F)}\delta$ is admissible representation of
$G(F)$. 
\end{prop}

\begin{prop}\label{prop:irreduciblity-of-compactly-induced-rep}
If $r\geq 2e$, then $\text{\rm ind}_{G(O_F)}^{G(F)}\delta$ is
irreducible representation of $G(F)$.
\end{prop}

We will prove these two propositions in the following subsections.

\subsection[]{}
The field extension $K/F$ is tamely ramified. Fix a prime element 
$\varpi_K\in O_K$ such that $\varpi_K^e\in O_{K_0}$ where $e$ is the
ramification index of $K/F$ and $K_0$ is the maximal unramified
subextension of $K/F$. The field $K$ is identified with a
$F$-subalgebra of the matrix algebra $M_n(F)$ by means of the regular
representation with respect to an $O_F$-basis of $O_K$. 

Take a generator $\beta$ of $O_K$ as an $O_F$-algebra, and let us
denote by $\chi_{\beta}(t)\in O_F[t]$ the characteristic polynomial of 
$\beta\in O_K\subset M_n(O_F)$. Then Shintani \cite{Shintani1968}
shows the following proposition

\begin{prop}
\label{prop:properties-of-chracteristic-polynomial-of-beta}
\begin{enumerate}
\item $\chi_{\beta}(t)\npmod{\frak p}\in\Bbb F[t]$ is the minimal
      polynomial of $\beta\npmod{\frak p}\in M_n(\Bbb F)$,
\item $\chi_{\beta}(t)\npmod{\frak p}=p(t)^e$ with a polynomial 
      $p(t)\in\Bbb F[t]$ irreducible over $\Bbb F$,
\item $\chi_{\beta}(t)\npmod{\frak{p}^2}$ is an irreducible 
      polynomial over the ring $O_F/\frak{p}^2$.
\end{enumerate}
\end{prop}

\begin{rem}
\label{rem:implication-of-property-of-characteristic-polynomial-of-beta}
The first statement of Proposition
\ref{prop:properties-of-chracteristic-polynomial-of-beta} implies that
\begin{enumerate}
\item for any $m>0$ and $X\in M_n(O_F)$, if 
      $X\beta\equiv\beta X\npmod{\frak{p}^m}$, then there exists a
  polynomial $f(t)\in O_F[t]$ such that 
  $X\equiv f(\beta)\npmod{\frak{p}^m}$,
\item $\{X\in M_n(O_F)\mid X\beta=\beta X\}=O_K$, 
\item for any $X\in M_n(O_F)$, if $\chi_X(t)=\chi_{\beta}(t)$, then 
      there exists $g\in GL_n(O_F)$ such that $X=g\beta g^{-1}$.
\end{enumerate}
\end{rem}

We have the Cartan decomposition 
\begin{equation}
 G(F)=\bigsqcup_{m\in\Bbb M}G(O_F)\varpi^mG(O_F)
\label{eq:cartan-decomposition-of-sl(n)}
\end{equation}
where 
$$
 \Bbb M
 =\left\{m=(m_1,m_2,\cdots,m_n)\in\Bbb Z^n\biggm|
   \begin{array}{l}
    m_1\geq m_2\geq\cdots\geq m_n,\\
    m_1+m_2+\cdots+m_n=0
   \end{array}\right\}
$$
and
$$
 \varpi^m=\begin{bmatrix}
           \varpi^{m_1}&      &            \\
                       &\ddots&            \\
                       &      &\varpi^{m_n}
          \end{bmatrix}
 \;\text{\rm for}\;
 m=(m_1,\cdots,m_n)\in\Bbb M.
$$
Since the restriction to $G(\frak{p}^l/\frak{p}^r)$ of 
our irreducible representation $\delta$ of $G(O_F/\frak{p}^r)$
contains the character $\psi_{\beta}$, we have
\begin{equation}
 \delta|_{G(\frak{p}^l/\frak{p}^r)}
 =\left(\bigoplus_{\dot g}g\ast\psi_{\beta}\right)^b
\label{eq:clifford-decomposition-of-induced-rep}
\end{equation}
with some integer $b>0$. Here 
$(g\ast\psi_{\beta})(x)=\psi_{\beta}(g^{-1}xg)$ is the conjugate of
$\psi_{\beta}$ and $\bigoplus_{\dot g}$ is the direct sum over 
$\dot g\in G(O_F/\frak{p}^r)/G(O_F/\frak{p}^r,\psi_{\beta})$ where
\begin{align}
 G(O_F/\frak{p}^r,\psi_{\beta})
 &=\{g\in G(O_F/\frak{p}^r)\mid g\ast\psi_{\beta}=\psi_{\beta}\}
    \nonumber\\
 &=\left\{\overline g\in G(O_F/\frak{p}^r)\mid 
      \text{\rm Ad}(g)\beta\equiv\beta\npmod{\frak{p}^{l^{\prime}}}
          \right\}
\label{eq:isotropy-subgroup-of-psi-beta}
\end{align}
is the isotropy subgroup of $\psi_{\beta}$. The second equality is due
to that fact 
$\overline g\ast\psi_{\beta}=\psi_{\text{\rm Ad}(g)\beta}$ for 
$g\in G(O_F)$.

\subsection[]{}
\label{subsec:proof-of-admissibility-of-induced-rep}
The proof of Proposition
\ref{prop:admissibility-of-compactly-induced-rep}. For any integer
$a>0$, put $G(\frak{p}^a)=G(O_F)\cap(1_n+\varpi^aM_n(O_F))$. We will
prove that the dimension of the space of the $G(\frak{p}^a)$-fixed 
vectors is finite. The Cartan decomposition 
\eqref{eq:cartan-decomposition-of-sl(n)} gives
$$
 G(F)=\bigsqcup_{s\in S}G(\frak{p}^a)sG(O_F)
$$
with 
$$
 S=\{k\varpi^m\mid \dot k\in G(\frak{p}^a)\backslash G(O_F), 
                   m\in\Bbb M\}.
$$
Then we have
$$
 \left.\text{\rm ind}_{G(O_F)}^{G(F)}\delta\right|_{G(\frak{p}^a)}
 =\bigoplus_{s\in S}
  \text{\rm ind}_{G(\frak{p}^a)\cap sG(O_F)s^{-1}}^{G(\frak{p}^a)}
   \delta^s
$$
with $\delta^s(h)=\delta(s^{-1}hs)$ 
($h\in G(\frak{p}^a)\cap sG(O_F)s^{-1}$). 
The Frobenius reciprocity gives
$$
 \text{\rm Hom}_{G(\frak{p}^a)}\left(\text{\bf 1},
    \text{\rm ind}_{G(O_F)}^{G(F)}\delta\right)
  =\bigoplus_{s\in S}\text{\rm Hom}_{s^{-1}G(\frak{p}^a)s\cap G(O_F)}
      (\text{\bf 1},\delta).
$$
Here $\text{\bf 1}$ is the one-dimensional trivial representation of 
$G(\frak{p}^a)$. If 
$$
 \text{\rm Hom}_{G(\frak{p}^a)}\left(\text{\bf 1},
  \text{\rm ind}_{G(O_F)}^{G(F)}\delta\right)\neq 0
$$
then there exists a 
$$
 s=k\varpi^m\in S
 \quad
 (k\in G(O_F), m=(m_1,\cdots,m_n)\in\Bbb M)
$$ 
such that 
$\text{\rm Hom}_{s^{-1}G(\frak{p}^a)s\cap G(O_F)}(\text{\bf 1},
  \delta)\neq 0$. 
If 
$$
 \text{\rm Max}\{m_i-m_{i+1}\mid 1\leq i<n\}=m_i-m_{i+1}\geq a
$$
then $\varpi^mU_i(O_F)\varpi^{-m}\subset G(\frak{p}^a)$ where
$$
 U_i=\left\{\begin{bmatrix}
             1_i&B\\
              0 &1_{n-1}
            \end{bmatrix}\biggm| B\in M_{i,n-i}\right\}
$$
is the unipotent part of the maximal parabolic subgroup
$$
 P_i=\left\{\begin{bmatrix}
             A&B\\
             0&D
            \end{bmatrix}\in G\biggm|
              A\in GL_i, D\in GL_{n-1}\right\}.
$$
So we have $U_i(O_F)\subset s^{-1}G(\frak{p}^a)s\cap G(O_F)$ so that
$$
 \text{\rm Hom}_{U_i(\frak{p}^l)}(\text{\bf 1},\delta)
 \supset
 \text{\rm Hom}_{s^{-1}G(\frak{p}^a)s\cap G(O_F)}
   (\text{\bf 1},\delta)\neq 0
$$
where $U_i(\frak{p}^l)=U_i(O_F)\cap G(\frak{p}^l)$. Then the
decomposition \eqref{eq:clifford-decomposition-of-induced-rep} implies that
there exists a $g\in G(O_F)$ such that 
$\psi_{\text{\rm Ad}(g)\beta}(h)=1$ for all $h\in U_i(\frak{p}^l)$,
that is
$$
 \tau\left(\varpi^{-l^{\prime}}
   \text{\rm tr}(g\beta g^{-1}\begin{bmatrix}
                               0&B\\
                               0&0
                              \end{bmatrix})\right)=0
$$
for all $B\in M_{i,n-1}(O_F)$. This means 
$$
 g\beta g^{-1}\equiv\begin{bmatrix}
                     A&\ast\\
                     0&D
                    \end{bmatrix}\npmod{\frak{p}^{l^{\prime}}}
$$
with $A\in M_i(O_F), D\in M_{n-i}(O_F)$, that is
$$
 \chi_{\beta}(t)\equiv\det(t1_i-A)\cdot\det(t1_{n-i}-D)
  \npmod{\frak{p}^{l^{\prime}}}.
$$
If $K/F$ is unramified, this is contradict against 2) of Proposition 
\ref{prop:properties-of-chracteristic-polynomial-of-beta}. If
$K/F$ is ramified, then $r\geq 4$ and $l^{\prime}\geq 2$ and a
contradiction to 3) of the proposition. So we have 
$$
 \text{\rm Max}\{m_i-m_{i+1}\mid 1\leq i<n\}<a.
$$
This implies that the number of $s\in S$ such that 
$\text{\rm Hom}_{s^{-1}G(\frak{p}^a)s\cap G(O_F)}
  (\text{\bf 1},\delta)\neq 0$ is finite, and then
$$
 \dim_{\Bbb C}\text{\rm Hom}_{G(\frak{p}^a)}\left(
   \text{\bf 1},\text{\rm ind}_{G(O_F)}^{G(F)}\delta\right)
 <\infty.
$$
This proves Proposition
\ref{prop:admissibility-of-compactly-induced-rep}.

\subsection[]{}
In this subsection, we will prove Proposition 
\ref{prop:irreduciblity-of-compactly-induced-rep}. To begin with, we
will prove the following proposition

\begin{prop}\label{prop:minimal-k-type-of-induced-rep}
If $r\geq 2e$, then
\begin{enumerate}
\item $\dim_{\Bbb C}\text{\rm Hom}_{G(O_F)}\left(
        \delta,\text{\rm ind}_{G(O_F)}^{G(F)}\delta\right)=1$,
\item if $\Delta$ is an irreducible representation of $G(O_F)$ which
  factors through the canonical surjection $G(O_F)\to
  G(O_F/\frak{p}^r)$ such that
$$
 \text{\rm Hom}_{G(O_F)}\left(
   \Delta,\text{\rm ind}_{G(O_F)}^{G(F)}\delta\right)\neq 0,
$$
  then $\Delta=\delta$.
\end{enumerate}
\end{prop}
\begin{proof}
Cartan decomposition \eqref{eq:cartan-decomposition-of-sl(n)} gives
$$
 \left.\text{\rm ind}_{G(O_F)}^{G(F)}\delta\right|_{G(O_F)}
 =\bigoplus_{m\in\Bbb M}
   \text{\rm ind}_{G(O_F)\cap\varpi^mG(O_F)\varpi^{-m}}^{G(O_F)}
    \delta^{\varpi^m}.
$$
Then Frobenius reciprocity gives
$$
 \text{\rm Hom}_{G(O_F)}\left(\Delta,
  \text{\rm ind}_{G(O_F)}^{G(F)}\delta\right)
 =\bigoplus_{m\in\Bbb M}
   \text{\rm Hom}_{\varpi^{-m}G(O_F)\varpi^m\cap G(O_F)}
    \left(\Delta^{\varpi^{-m}},\delta\right).
$$
Now take a $m=(m_1,\cdots,m_n)\in\Bbb M$ such that
$$
 \text{\rm Hom}_{\varpi^{-m}G(O_F)\varpi^m\cap G(O_F)}
    \left(\Delta^{\varpi^{-m}},\delta\right)\neq 0.
$$
Suppose 
$\text{\rm Max}\{m_i-m_{i+1}\mid 1\leq i<n\}=m_i-m_{i+1}\geq 2$. Then 
$$
 U_i(O_F)=\varpi^{-m}U_i(O_F)\varpi^m\cap U_i(O_F)
 \subset\varpi^{-m}G(O_F)\varpi^m\cap G(O_F)
$$
and we have
$$
 \text{\rm Hom}_{U_i(\frak{p}^{r-2})}
  \left(\Delta^{\varpi^{-m}},\delta\right)
 \supset
 \text{\rm Hom}_{\varpi^{-m}G(O_F)\varpi^m\cap G(O_F)}
  \left(\Delta^{\varpi^{-m}},\delta\right)\neq 0.
$$
Because $G(\frak{p}^r)\subset\text{\rm Ker}\,\Delta$ and 
$\varpi^mU_i(\frak{p}^{r-2})\varpi^{-m}\subset U_i(\frak{p}^r)$, we have 
$$
 \text{\rm Hom}_{U_i(\frak{p}^{r-k})}(\text{\bf 1},\delta)
 \supset
 \text{\rm Hom}_{U_i(\frak{p}^{r-2})}(\text{\bf 1},\delta)\neq 0.
$$
Here $\text{\bf 1}$ is the trivial one-dimensional representation of
$U_i(\frak{p}^{r-2})$. On the other hand $r\geq 2e$ implies 
$r-e\geq l$ and $U_i(\frak{p}^{r-2})\subset G(\frak{p}^l)$. Then due
to the decomposition \eqref{eq:clifford-decomposition-of-induced-rep},
there exists a $g\in G(O_F)$ such that 
$\psi_{\text{\rm Ad}(g)\beta}(h)=1$ for all $h\in U_i(\frak{p}^{r-k})$
with
$$
 k=\begin{cases}
    1&: e=1,\\
    2&: e>1.
   \end{cases}
$$
This means 
$$
 g\beta g^{-1}\equiv\begin{bmatrix}
                     A&\ast\\
                     0&D
                    \end{bmatrix}\npmod{\frak{p}^k}
$$
with some $A\in M_i(O_F)$ and $D\in M_{n-i}(O_F)$. Then
$$
 \chi_{\beta}(t)\equiv
  \det(t1_i-A)\cdot\det(t1_{n-i}-D)\npmod{\frak{p}^k}
$$
which contradicts to Proposition
\ref{prop:properties-of-chracteristic-polynomial-of-beta}. So, if 
$$
 \text{\rm Hom}_{\varpi^{-m}G(O_F)\varpi^m\cap G(O_F)}
  \left(\Delta^{\varpi^{-m}},\delta\right)\neq 0
$$
then $0\leq m_i-m_{i+1}\leq 1$ for all $1\leq i<n$. If there exists
$1\leq i<n$ such that $m_i-m_{i+1}=1$, then we have as above 
$$
 g\beta g^{-1}\equiv\begin{bmatrix}
                     A&\ast\\
                     0&D
                    \end{bmatrix}\npmod{\frak p}
$$
with some $g\in G(O_F)$, $A\in M_i(O_F)$ and $D\in M_{n-i}(O_F)$, and
hence
$$
 \chi_{\beta}(t)\equiv
  \det(t1_i-A)\cdot\det(t1_{n-i}=D)\npmod{\frak p}.
$$
Then 2) of  Proposition
\ref{prop:properties-of-chracteristic-polynomial-of-beta} implies that 
$i=\deg\det(t1_i-A)$ is a multiple of $f$. So we have 
$m_1-m_n<e\leq l^{\prime}$. We can take a $\gamma\in\frak{g}(O_F)$
such that
$$
 \Delta|_{G(\frak{p}^l/\frak{p}^r)}
 =\left(\bigoplus_{\dot h}h\ast\psi_{\gamma}\right)^c
$$
with some integer $c>0$. Here $\bigoplus_{\dot h}$ is the direct sum 
over $\dot h\in
G(O_F/\frak{p}^r)/G(O_F/\frak{p}^r,\psi_{\gamma})$. Then we have
$$
 \bigoplus_{\dot h}\bigoplus_{\dot g}
  \text{\rm Hom}_{G(\frak{p}^{l+m_1-m_n})}
   \left((h\ast\psi_{\gamma})^{\varpi^{-m}},g\ast\psi_{\beta}\right)
 \neq 0
$$
because 
$G(\frak{p}^{l+m_1-m_n})\subset\varpi^{-m}G(O_F)\varpi^m\cap G(O_F)$
and 
$\varpi^mG(\frak{p}^{l+m_1-m_n})\varpi^{-m}\subset
G(\frak{p}^l)$. Hence there exist $g,h\in G(O_F)$ such that
$$
 \psi_{\text{\rm Ad}(g)\beta}(x)
 =\psi_{\text{\rm Ad}(h)\gamma}(\varpi^mx\varpi^{-m})
$$
for all $x\in G(\frak{p}^{l+m_1-m_n})$. This means 
$$
 g\beta g^{-1}\equiv\varpi^{-m}h\gamma h^{-1}\varpi^m
 \npmod{\frak p},
$$
that is $\varpi^mg\beta g^{-1}\varpi^{-m}\in M_n(O_F)$. Then there
exists a $g^{\prime}\in GL_n(O_F)$ such that 
$\varpi^mg\beta g^{-1}\varpi^{-m}=g^{\prime}\beta g^{\prime -1}$ and
hence $g^{\prime -1}\varpi^mg\in K$ due to 2) and 3) of Remark
\ref{rem:implication-of-property-of-characteristic-polynomial-of-beta}. 
On the other hand we have 
$$
 N_{K/F}(g^{\prime -1}\varpi^mg)
 =\det(g^{\prime -1}\varpi^mg)\in O_F^{\times}
$$
so that 
$g^{\prime -1}\varpi^mg\in O_K^{\times}\subset GL_n(O_F)$. Hence 
$m=(0,\cdots,0)$. Now we have proved
$$
 \text{\rm Hom}_{G(O_F)}
  \left(\Delta,\text{\rm ind}_{G(O_F)}^{G(F)}\delta\right)
 =\text{\rm Hom}_{G(O_F)}(\Delta,\delta)
$$
which implies the two statements of the proposition.
\end{proof}

The proof of Proposition
\ref{prop:irreduciblity-of-compactly-induced-rep}. Let 
$W\subset\text{\rm ind}_{G(O_F)}^{G(F)}\delta$ be a non-trivial
$G(F)$-stable subspace. Then we have by Frobanius reciprocity 
\begin{align*}
 0\neq
 \text{\rm Hom}_{G(F)}
   \left(W,\text{\rm ind}_{G(O_F)}^{G(F)}\delta\right)
 &\subset
  \text{\rm Hom}_{G(O_F)}
   \left(W,\text{\rm ind}_{G(O_F)}^{G(F)}\delta\right)\\
 &=\text{\rm Hom}_{G(O_F)}(W,\delta),
\end{align*}
and hence $W$ contains $\delta$ as $G(O_F)$-module. On the other hand,
also by Frobenius reciprocity, we have
$$
 0\neq\text{\rm Hom}_{G(F)}(V,V/W)
 =\text{\rm Hom}_{G(O_F)}(\delta,V/W)
$$
with $V=\text{\rm ind}_{G(O_F)}^{G(F)}\delta$, and hence $V/W$
contains $\delta$ as a $G(O_F)$-module. Since the admissible
representations of $G(F)$ are semi-simple over the compact subgroup
$G(O_F)$, this means that $\text{\rm ind}_{G(O_F)}^{G(F)}\delta$
contains $\delta$ with multiplicity as least two, which contradicts to
1) of Proposition \ref{prop:minimal-k-type-of-induced-rep}. So the
$G(F)$-stable subspace of $\text{\rm ind}_{G(O_F)}^{G(F)}\delta$ is 
trivial. 

\section{Formal degree of supercuspidal representations}
\label{sec:formal-degree-of-supercuspidal-representation}

In this section, we will prove Proposition
\ref{prop:dimension-of-delta}. 

\subsection[]{}
\label{subsec:explicit-description-of-delta}
We assume that $n$ is prime to $p$ so that
\begin{itemize}
\item[I)] the trace form
$$
 \frak{g}(\Bbb F)\times\frak{g}(\Bbb F)\to\Bbb F
 \quad
 ((X,Y)\mapsto\text{\rm tr}(XY))
$$
          is non-degenerate.
\end{itemize}
For any $X\in M_n(O_F)$, we have
$$
 \det(1_n+\varpi^lX)\equiv 1+\varpi^l\text{\rm tr}\,X
  \npmod{\frak{p}^{2l}}.
$$
On the other hand we have 
$$
 \det g\equiv
 1+\varpi^{l-1}\text{\rm tr}\,X
  +2^{-1}\varpi^{2l-2}\left(\text{\rm tr}\,X\right)^2
    \npmod{\frak{p}^{2l-1}}
$$
for $g=1_n+\varpi^{l-1}X+2^{-1}\varpi^{2l-2}X^2\in GL_n(O_F)$ so that
\begin{itemize}
\item[II)] for any $r=l+l^{\prime}$ with $0<l^{\prime}\leq l$, we have
  a group isomorphism
$$
 \frak{g}(O_F/\frak{p}^{l^{\prime}})\,\tilde{\to}\,
 G(\frak{p}^l/\frak{p}^r)
 \quad
 (X\npmod{\frak{p}^{l^{\prime}}}\mapsto 
  1_n+\varpi^lX\npmod{\frak{p}^r}),
$$
\item[III)] if $r=2l-1>1$ is odd, then we have a map
$$
 \frak{g}(O_F)\to G(\frak{p}^{l-1}/\frak{p}^r)
 \quad
 (X\mapsto 1_n+\varpi^{l-1}X+2^{-1}\varpi^{2l-2}X^2\npmod{\frak{p}^r}).
$$
\end{itemize}
Then the general theory developed by \cite{Takase2019} is applicable
to our case. Let us recall the general theory by writing down
explicitly the irreducible 
representation $\delta$ of $G(O_F/\frak{p}^r)$ corresponding to the
character
$$
 \theta:G_{\beta}(O_F/\frak{p}^r)\to\Bbb C^{\times}
$$
such that $\theta=\psi_{\beta}$ on 
$G_{\beta}(O_F/\frak{p}^r)\cap G(\frak{p}^l/\frak{p}^r)$. The general theory
of \cite{Takase2019} says that 
\begin{equation}
 \delta
 =\text{\rm Ind}_{G(O_F/\frak{p}^r,\psi_{\beta})}^{G(O_F/\frak{p}^r)}
   \sigma_{\beta,\theta}
\label{eq:delta-is-induced-from-sigma-beta-theta}
\end{equation}
with an irreducible representation 
$\sigma_{\beta,\theta}$ of the isotropy subgroup 
$G(O_F/\frak{p}^r,\psi_{\beta})$ of $\psi_{\beta}$ such that
\begin{equation}
 \dim\sigma_{\beta,\theta}
 =\begin{cases}
   1&:\text{\rm $r$ is even}, \\
   q^{n(n-1)/2}&:\text{\rm $r$ is odd}.
  \end{cases}
\label{eq:dimension-of-sigma-beta-theta}
\end{equation}
Because the canonical group homomorphism 
$G(O_F/\frak{p}^r)\to G(O_F/\frak{p}^{l^{\prime}})$ is surjective, 
\eqref{eq:isotropy-subgroup-of-psi-beta} implies 
\begin{equation}
 G(O_F/\frak{p}^r,\psi_{\beta})
 =G_{\beta}(O_F/\frak{p}^r)\cdot G(\frak{p}^{l^{\prime}}/\frak{p}^r).
\label{eq:structure-of-isotropy-subgroup}
\end{equation}
If $r=2l$ is even, then $\sigma_{\beta,\theta}$ is defined by
$$
 \sigma_{\beta,\theta}(gh)=\theta(g)\cdot\psi_{\beta}(h)
$$
for $g\in G_{\beta}(O_F/\frak{p}^r)$ and 
$h\in G(\frak{p}^l/\frak{p}^r)$. 

If $r=2l-1$ is odd, then $\sigma_{\beta,\theta}$ is realized on the
complex vector space $L^2(\Bbb W^{\prime})$ of the complex valued
functions on $\Bbb W^{\prime}$ where we fix a polarization 
$\Bbb V_{\beta}=\Bbb W\oplus\Bbb W^{\prime}$ of the symplectic space 
$\Bbb V_{\beta}$ over $\Bbb F$. The first statement of 
Remark
\ref{rem:implication-of-property-of-characteristic-polynomial-of-beta}
implies 
\begin{align*}
 \frak{g}_{\beta}(\Bbb F)
 &=\{X\in\Bbb F[\overline\beta]\subset M_n(\Bbb F)\mid
      \text{\rm tr}\,X=0\}\\
 &=\{\overline X\in O_K/\frak{p}_K^e\mid 
      T_{K/F}(X)\equiv 0\npmod{\frak p}\}
\end{align*}
with $\overline\beta=\beta\npmod{\frak p}\in M_n(\Bbb F)$. Then we
have
$$
 \dim_{\Bbb F}\Bbb V_{\beta}=(n^2-1)-(n-1)=n(n-1)
$$
so that $\dim_{\Bbb C}L^2(\Bbb W^{\prime})=q^{n(n-1)/2}$. Let 
$H_{\beta}$ be the Heisenberg group associated with the symplec space
$\Bbb V_{\beta}$ over $\Bbb F$, that is 
$H_{\beta}=\Bbb V_{\beta}\times\Bbb C^{\times}$ with a group operation 
$$
 (u,s)\cdot(v,t)
 =(u+v,st\widehat{\tau}\left(2^{-1}\langle u,v\rangle_{\beta}\right))
$$
and $\pi_{\beta}$ the Schr\"odinger representation of $H_{\beta}$ on 
$L^2(\Bbb W^{\prime})$, that is
$$
 \left(\pi_{\beta}(u,s)f\right)(w)
 =s\cdot\widehat{\tau}\left(
    2^{-1}\langle u_-,u_+\rangle_{\beta}+\langle w,u_+\rangle_{\beta}
      \right)\cdot f(w+u_-)
$$
for $(u,s)\in H_{\beta}$ and $f\in L^2(\Bbb W^{\prime})$ with 
$u=u_-+u_+$ ($u_-\in\Bbb W^{\prime}, u_+\in\Bbb W$). 

A representation $\pi_{\beta,\theta}$ of 
$G(\frak{p}^{l-1}/\frak{p}^r)$ on $L^2(\Bbb W^{\prime})$ is defined as
follows. 
Take a 
$h=1_n+\varpi^{l-1}T\npmod{\frak{p}^r}
 \in G(\frak{p}^{l-1}/\frak{p}^r)$ with $T\in M_n(O_F)$. 
Then $T\npmod{\frak{p}^{l-1}}\in\frak{g}(O_F/\frak{p}^{l-1})$ and the
image of it under the canonical surjection 
$\frak{g}(O_F/\frak{p}^{l-1})\to\frak{g}(\Bbb F)$ is denoted by 
$\widehat T$. 
Put $v=\widehat T\npmod{\frak{g}_{\beta}(\Bbb F)}\in\Bbb V_{\beta}$
and $Y=\widehat T-[v]\in\frak{g}_{\beta}(\Bbb F)$. Then
$$
 \pi_{\beta,\theta}(h)
 =\tau\left(\varpi^{-l}\text{\rm tr}(T\beta)
            -2^{-1}\varpi^{-1}\text{\rm tr}(T^2\beta)\right)\cdot
  \rho(Y)\cdot\pi_{\beta}(v,1).
$$
Here an additive character 
$\rho:\frak{g}_{\beta}(\Bbb F)\to\Bbb C^{\times}$ is defined as
follows. Since the $O_F$-group scheme $G_{\beta}$ is smooth, the
canonical map $\frak{g}_{\beta}(O_F)\to\frak{g}_{\beta}(\Bbb F)$ is
surjective. So for any $Y\in\frak{g}_{\beta}(\Bbb F)$, we can take a 
$X\in\frak{g}_{\beta}(O_F)$ such that $Y=X\npmod{\frak p}$. Then 
$$
 g=1_n+\varpi^{l-1}X+2^{-1}\varpi^{2l-2}X^2\npmod{\frak{p}^r}
 \in G_{\beta}(\frak{p}^{l-1}/\frak{p}^r)
$$
and we will define
$$
 \rho(Y)
 =\tau\left(-\varpi^{-l}\text{\rm tr}(X\beta)\right)\cdot
  \theta(g).
$$
There exists a group homomorphism 
$U:G_{\beta}(O_F/\frak{p}^r)\to GL_{\Bbb C}(L^2(\Bbb W^{\prime}))$
such that 
$$
 \pi_{\beta,\theta}(g^{-1}hg)
 =U(g)^{-1}\circ\pi_{\beta,\theta}(h)\circ U(g)
$$
for all $g\in G_{\beta}(O_F/\frak{p}^r)$ and 
$h\in G(\frak{p}^{l-1}/\frak{p}^r)$. Now the representation 
$\sigma_{\beta,\theta}$ of 
$$
 G(O_F/\frak{p}^r,\psi_{\beta})
 =G_{\beta}(O_F/\frak{p}^r)\cdot G(\frak{p}^{l-1}/\frak{p}^r)
$$
is defined by
$$
 \sigma_{\beta,\theta}(gh)
 =\theta(g)\cdot U(g)\circ\pi_{\beta,\theta}(h)
$$
for $g\in G_{\beta}(O_F/\frak{p}^r)$ and 
$h\in G(\frak{p}^{l-1}/\frak{p}^r)$. 

\subsection[]{}
\label{subsec:dimension-of-delta}
The proof of Proposition \ref{prop:dimension-of-delta}. By 
\eqref{eq:delta-is-induced-from-sigma-beta-theta}, we have
\begin{equation}
 \dim\delta
 =(G(O_F/\frak{p}^r):G(O_F/\frak{p}^r,\psi_{\beta}))\cdot
  \dim\sigma_{\beta,\theta}.
\label{eq:dim-delta-is-index-times-dim-sigma-beta-theta}
\end{equation}
Because of \eqref{eq:structure-of-isotropy-subgroup}, we have
$$
 |G(O_F/\frak{p}^r,\psi_{\beta})|
 =\frac{|G_{\beta}(O_F/\frak{p})||G(\frak{p}^{l^{\prime}}/\frak{p}^r)|}
       {|G_{\beta}(O_F/\frak{p}^r)\cap G(\frak{p}^{l-1}/\frak{p}^r)|}.
$$
$G(\frak{p}^{l^{\prime}}/\frak{p}^r)$ is the kernel of the canonical
surjection $G(O_F/\frak{p}^r)\to G(O_F/\frak{p}^{l^{\prime}})$, and 
$G_{\beta}(O_F/\frak{p}^r)\cap G(\frak{p}^{l^{\prime}}/\frak{p}^r)$ is
the kernel of the canonical surjection 
$G_{\beta}(O_F/\frak{p}^r)\to
 G_{\beta}(O_F/\frak{p}^{l^{\prime}})$. Hence we have 
\begin{equation}
 (G(O_F/\frak{p}^r):G(O_F/\frak{p}^r,\psi_{\beta}))
 =\frac{|G(O_F/\frak{p}^{l^{\prime}})|}
       {|G_{\beta}(O_F/\frak{p}^{l^{\prime}})|}.
\label{eq:group-index-is-ratio-of-g-and-g-beta}
\end{equation}
We have
\begin{equation}
 |G(O_F/\frak{p}^{l^{\prime}})|
 =q^{l^{\prime}n(n-1)/2}\prod_{k=2}^n(1-q^{-k}).
\label{eq:order-of-sl-n-o-mod-p-l-prime}
\end{equation}
On the other hand $G_{\beta}(O_F/\frak{p}^{l^{\prime}})$ is the kernel
of 
$$
 \left(O_K/\frak{p}_K^{el^{\prime}}\right)^{\times}
 \to
 \left(O_F/\frak{p}^{l^{\prime}}\right)
 \quad
 (\varepsilon\mapsto N_{K/F}(\varepsilon)),
$$
and $1+\frak{p}=N_{K/F}(1+\frak{p}_K)$ 
(see \cite[p.32, Prop.2]{Cassels-Frohlich}). Then we have
\begin{equation}
 |G_{\beta}(O_F/\frak{p}^{l^{\prime}})|
 =(O_F^{\times}:N_{K/F}(O_K^{\times}))\cdot q^{l^{\prime}(n-1)}\cdot
   \frac{1-q^{-f}}
        {1-q^{-1}}.
\label{eq:order-of-g-beta-o-mod-p-l-prime}
\end{equation}
Combining the equations 
\eqref{eq:dimension-of-sigma-beta-theta}, 
\eqref{eq:dim-delta-is-index-times-dim-sigma-beta-theta},
\eqref{eq:group-index-is-ratio-of-g-and-g-beta},
\eqref{eq:order-of-sl-n-o-mod-p-l-prime} and 
\eqref{eq:order-of-g-beta-o-mod-p-l-prime}, the proof of Proposition 
\ref{prop:dimension-of-delta} is completed. 

\section{Induced representations of Weil group}
\label{sec:induced-representation-of-weil-group}

In this section, we will assume that $K/F$ is a tamely ramified 
Galois extension and
will prove Theorem \ref{th:ratio-of-gamma-factor}. 

The algebraic extensions of $F$ are taken within a fixed algebraic
closure $\overline F$ of $F$. Define a group homomorphism
$$
 \nu_F:{\overline F}^{\times}\to\Bbb Q
$$
by $\nu_F(x)=(F(x):F)^{-1}\text{\rm ord}_F(N_{F(x)/F}(x))$ 
($0\neq x\in\overline F$) and put $\nu_F(0)=\infty$. For an algebraic
extension $L/F$, put
$$
 O_L=\{x\in L\mid\nu_F(x)\geq 0\},
 \quad
 \frak{p}_L=\{x\in L\mid\nu_F(x)>0\}.
$$
Then $\Bbb L=O_L/\frak{p}_L$ is an algebraic extension of 
$\Bbb F=O_F/\frak{p}$. 
Let us denote by $F^{\text{\rm sep}}$ the separable closure of $F$.

\subsection[]{}
To begin with, we will recall the definition of the $L$-factor, 
the $\varepsilon$-factor and the $\gamma$-factor of a representation of
a Weil group (or Weil-Deligne group, more precisely). See
\cite{GrossReeder2010} for the details.  

 The Weil group $W_F$ of $F$ is the inverse image by the canonical
restriction mapping 
$$
 \text{\rm Gal}(F^{\text{\rm sep}}/F)\to
 \text{\rm Gal}(F^{\text{\rm ur}}/F)
$$
of the cyclic subgroup 
$\langle\text{\rm Fr}\rangle
 \subset\text{\rm Gal}(F^{\text{\rm ur}}/F)$ generated by the geometric
 Frobanius automorphism $\text{\rm Fr}$ which induces the inverse of
 the Frobanius automorphism in 
$\text{\rm Gal}(\overline{\Bbb F}/\Bbb F)$. Fix an extension 
$\widetilde{\text{\rm Fr}}\in\text{\rm Gal}(F^{\text{\rm sep}}/F)$ of
 $\text{\rm Fr}$. Put
$$
 I_F=\text{\rm Gal}(F^{\text{\rm sep}}/F^{\text{\rm ur}})
$$
which is a normal subgroup of $W_F$ and we have
$W_F=\langle\widetilde{\text{\rm Fr}}\rangle\ltimes I_F$. Weil group
 $W_F$ is endowed with the topology such that $I_F$, with the usual
 Krull topology, is an open compact subgroup of $W_F$. 
 
Take a complex linear algebraic group $\mathcal{G}$ such that its
connected component $\mathcal{G}^o$ is reductive. There is a bijective
correspondence between the conjugacy classes of the triplets 
$(\rho,\,\mathcal{G},N)$ such that
\begin{enumerate}
\item $\rho:W_F\to\mathcal{G}$ is a group homomorphism which is
      continuous on $I_F$, 
\item $\rho(\widetilde{\text{\rm Fr}})\in\mathcal{G}$ is a semi-simple
      element, 
\item $N\in\text{\rm Lie}(\mathcal{G})$ is a nilpotent element such
      that $\rho(\sigma)N=|\sigma|_FN$ for all $\sigma\in W_F$
\end{enumerate}
and the conjugacy classes of the continuous group homomorphisms 
$$
 \varphi:W_F\times SL_2(\Bbb C)\to\mathcal{G}
$$
such that
\begin{enumerate}
\item $\text{\rm Ker}(\varphi)\cap I_F$ is an open subgroup of $I_F$, 
\item $\varphi(\widetilde{\text{\rm Fr}})\in\mathcal{G}$ is
      semi-simple element,
\item $\varphi|_{SL_2(\Bbb C)}$ is a morphism of complex algebraic
  group
\end{enumerate}
defined by 
$$
 \rho|_{I_F}=\varphi|_{I_F},
 \quad
 \rho(\widetilde{\text{\rm Fr}})
 =\varphi(\widetilde{\text{\rm Fr}})\cdot
  \varphi\begin{pmatrix}
          q^{-1/2}&0\\
          0&q^{1/2}
         \end{pmatrix},
 \quad
 N=d\varphi\begin{pmatrix}
            0&1\\
            0&0
           \end{pmatrix}
$$
where $d\varphi:\frak{sl}_2(\Bbb C)\to\text{\rm Lie}(\mathcal{G})$ is
the differential of $\varphi|_{SL_2(\Bbb C)}$. The triplet 
$(\rho,\mathcal{G},N)$ or the group homomorphism $\varphi$ is called a
representation of Weil-Deligne group on $\mathcal{G}$. 

Take an algebraic complex representation $(r,V)$ of $\mathcal{G}$,
that is $V$ is a finite dimensional complex vector space and 
$r:\mathcal{G}\to GL_{\Bbb C}(V)$ is a morphism of complex algebraic
  group. Since the kernel $V_N$ of 
$dr(N)\in\text{\rm Lie}(GL_{\Bbb C}(V))=\text{\rm End}_{\Bbb C}(V)$ is 
$r\circ\rho(I_F)$-stable, we will put
$$
 L(\varphi,r,s)
 =\det\left(
   1-q^{-s}r\circ\rho(\widetilde{\text{\rm Fr}})|_{V_N^{I_F}}
              \right)^{-1}
$$
where $V_N^{I_F}$ is the subspace of $V_N$ of the $I_F$-fixed
vectors. The $\varepsilon$-factor is defined by
$$
 \varepsilon(\varphi,r,s)
 =\varepsilon_0(r\circ\rho,s)\cdot
  \det\left(
   -q^{-s}r\circ\rho(\widetilde{\text{\rm Fr}})|_{V^{I_F}/V_N^{I_F}}
            \right)
$$
where $V^{I_F}$ is the subspace of $V$ of the $I_F$-fixed vectors and 
$$
 \varepsilon_0(r\circ\rho,s)
 =w(r\circ\rho)\cdot q^{-a(r\circ\rho)(s-1/2)}
$$ 
is the $\varepsilon$-factor of the representation 
$$
 r\circ\rho:W_F\to GL_{\Bbb C}(V)
$$
of Weil group. Here $w(r\circ\rho)$ is a complex number 
of absolute value one (the root
number) and $a(r\circ\rho)$ is the Artin conductor defined as follows. 
There exists a finite extension $L/F^{\text{\rm ur}}$ such that 
$\text{\rm Gal}(F^{\text{sep}}/L)\subset I_F\cap\text{\rm Ker}(\rho)$. Put
\begin{align*}
 D_t&=D_t(L/F^{\text{\rm ur}})\\
 &=\left\{\sigma\in\text{\rm Gal}(L/F^{\text{\rm ur}})\mid
    x^{\sigma}\equiv x\npmod{\frak{p}_L^{t+1}}\,
     \text{\rm for}\,\forall x\in O_L\right\}
\end{align*}
for $t=0,1,2,\cdots$. Then $a(r\circ\rho)$ is defined by
$$
 a(r\circ\rho)
 =\sum_{t=0}^{\infty}(D_0:D_t)^{-1}\dim_{\Bbb C}\left(V/V^{D_t}\right)
$$
where $V^{D_t}$ is the subspace of $V$ of the 
$r\circ\rho(\widetilde D_t)$-fixed vectors with the inverse image 
$\widetilde D_t$ of $D_t$ by the restriction mapping 
$I_F\to\text{\rm Gal}(L/F^{\text{\rm ur}})$. Finally the
$\gamma$-factor is defined by
$$
 \gamma(\varphi,r,s)
 =\varepsilon(\varphi,r,s)\cdot
  \frac{L(\varphi,r^{\vee},1-s)}
       {L(\varphi,r,s)}
$$
where $r^{\vee}$ is the dual of $r$. 

In the following discussions, the complex algebraic group
$\mathcal{G}$ is the $L$-group of $G=SL_n$ over $F$, that is 
$\mathcal{G}=PGL_n(\Bbb C)$ since $SL_n$ is split over $F$, 

\subsection[]{}
\label{subsec:principal-parameter-of-weil-group}
Let $\text{\rm Sym}_{n-1}$ be the symmetric tensor representation of
$SL_2(\Bbb C)$ on the space of the complex coefficient homogeneous
polynomials of $X,Y$ of degree $n-1$, which gives the group
homomorphism 
$$
 \text{\rm Sym}_{n-1}:SL_2(\Bbb C)\to GL_n(\Bbb C)
$$
with respect to the $\Bbb C$-basis 
$$
 \{X^{n-1},X^{n-2}Y,\cdots,XY^{n-2},Y^{n-1}\}.
$$
Then 
$$
 d\text{\rm Sym}_{n-1}\begin{pmatrix}
                       0&1\\
                       0&0
                      \end{pmatrix}=N_0
 =\begin{bmatrix}
   0&1& &      & \\
    &0&1&      & \\
    & &\ddots&\ddots& \\
    & & &0&1\\
    & & &      &0
   \end{bmatrix}
$$
is the nilpotent element in $\frak{pgl}_n(\Bbb C)=\frak{sl}_n(\Bbb C)$
associated with the standard {\it \'epinglage} of the standard root
system of $\frak{sl}_n(\Bbb C)$. Then 
$$
 \varphi_0:W_F\times SL_2(\Bbb C)
           \xrightarrow{\text{\rm proj.}}SL_2(\Bbb C)
           \xrightarrow{\text{\rm Sym}_{n-1}}GL_n(\Bbb C)
           \xrightarrow{\text{\rm canonical}}PGL_n(\Bbb C)
$$
is a representation of Weil-Deligne group with the associated triplet 
$(\rho_0,PGL_n(\Bbb C),N_0)$ such that
$\rho_0|_{I_F}$ is trivial and
$$
 \rho_0(\widetilde{\text{\rm Fr}})
 =\overline{
   \begin{bmatrix}
    q^{-(n-1)/2}&            &      &           &           \\
                &q^{-(n-3)/2}&      &           &           \\
                &            &\ddots&           &           \\
                &            &      &q^{(n-3)/2}&           \\
                &            &      &           &q^{(n-1)/2}
   \end{bmatrix}}
 \in PGL_n(\Bbb C).
$$
Let $\text{\rm Ad}:PGL_n(\Bbb C)\to GL_{\Bbb C}(\widehat{\frak g})$ be 
the adjoint representation of $PGL_n(\Bbb C)$ on 
$\widehat{\frak g}=\frak{sl}_n(\Bbb C)$. Then 
\begin{equation}
 \left\{N_0^k\mid k=1,2,\cdots,n-1\right\}
\label{eq:c-basis-of-widehat-g-n-0-for-sl-n}
\end{equation}
is the $\Bbb C$-basis of $\widehat{\frak g}_{N_0}$. The representation
matrix of $\text{\rm Ad}\circ\rho_0(\widetilde{\text{\rm Fr}})$ on 
$\widehat{\frak g}_{N_0}$ with respect to the $\Bbb C$-basis 
\eqref{eq:c-basis-of-widehat-g-n-0-for-sl-n} is
$$
  \begin{bmatrix}
   q^{-1}&      &      &          \\
         &q^{-2}&      &          \\
         &      &\ddots&          \\
         &      &      &q^{-(n-1)}
  \end{bmatrix}
$$
so that we have
$$
 L(\varphi_0,\text{\rm Ad},s)
 =\prod_{k=1}^{n-1}\left(1-q^{-(s+k)}\right)^{-1}.
$$
On the other hand \cite[p.448]{GrossReeder2010} shows 
$$
 \varepsilon(\varphi_0,\text{\rm Ad},0)=q^{n(n-1)/2}.
$$
Since the symmetric tensor representation $\text{\rm Sym}_{n-1}$ is
self-dual, we have
\begin{equation}
 \gamma(\varphi_0)=\gamma(\varphi_0,\text{\rm Ad},0)
 =q^{n(n-1)/2}\cdot\frac{1-q^{-1}}
                        {1-q^{-n}}.
\label{eq:gamma-factor-of-principal-parametor}
\end{equation}

\subsection[]{}
\label{subsec:l-factor-of-induced-rep-of-weil-group}
In this subsection, we will compute $L(\varphi,\text{\rm Ad},s)$ for a
representation $\varphi$ induced from a character of
$K^{\times}$ in the situation a little more general than that of
subsection \ref{subsec:ratio-of-gamma-factors}. 

Let $K/F$ be a tamely ramified Galois extension of degree $n$ and 
$$
 \theta:K^{\times}\to\Bbb C^{\times}
$$
is a continuous character such that $x\mapsto\theta(x^{\sigma-1})$ is
the trivial character of $K^{\times}$ only if 
$\sigma\in\text{\rm Gal}(K/F)$ is $1$. Based upon the group extension 
\eqref{eq:relative-weil-group-is-group-extension-of-fundamental-class},
we have an identification 
$W_{K/F}=\text{\rm Gal}(K/F)\times K^{\times}$ with group operation
$$
 (\sigma,x)\cdot(\tau,y)
 =(\sigma\tau,x^{\tau}\cdot y\cdot\alpha_{K/F}(\sigma,\tau))
$$
with the fundamental class 
$[\alpha_{K/F}]\in H^2(\text{\rm Gal}(K/F),K^{\times})$. Then the
induced representation $\text{\rm Ind}_{K^{\times}}^{W_{K/F}}\theta$
is realized on the complex vector space $V$ of the complex valued
functions on $\text{\rm Gal}(K/F)$ with the action of $W_{K/F}$
defined by
$$
 (\sigma\cdot\psi)(\tau)
 =\theta\left(\alpha_{K/F}(\sigma,\sigma^{-1}\tau)\right)\cdot
    \psi(\sigma^{-1}\tau),
 \qquad
 \alpha\cdot\psi=\theta_{\alpha}\cdot\psi
$$
for $(\sigma,\alpha)\in W_{K/F}$ and $\psi\in V$ with 
$\theta_{\alpha}\in V$ defined by 
$\theta_{\alpha}(\tau)=\theta(\alpha^{\tau})$. Take the standard basis
$\{\psi_{\sigma}\}_{\sigma\in\text{\rm Gal}(K/F)}$ of $V$ where
$$
 \psi_{\sigma}(\tau)=\begin{cases}
                      1&:\tau=\sigma,\\
                      0&:\tau\neq\sigma.
                     \end{cases}
$$
Then 
$$
 \rho\cdot\psi_{\sigma}
 =\theta\left(\alpha_{K/F}(\rho,\sigma)\right)\cdot\psi_{\rho\sigma}
$$
for $\rho,\sigma\in\text{\rm Gal}(K/F)$. In particular 
$\sigma\cdot\psi_1=\psi_{\sigma}$. We have

\begin{prop}
\label{prop:induced-rep-of-relative-weil-group-is-irreducible}
$\text{\rm Ind}_{K^{\times}}^{W_{K/F}}\theta$ is an irreducible
representation of $W_{K/F}$.
\end{prop}
\begin{proof}
Take a $T\in\text{\rm End}_{W_{K/F}}(V)$. If $(T\psi_1)(\tau)\neq 0$
with $\tau\in\text{\rm Gal}(K/F)$, then, for any 
$\alpha\in K^{\times}$, $\alpha\cdot T\psi_1=T(\alpha\cdot\psi_1)$
  implies $\theta(\alpha^{\tau})=\theta(\alpha)$, and hence
  $\tau=1$. This means $T\psi_1=c\psi_1$ with a $c\in\Bbb C$. Then we
  have
$T\psi_{\sigma}=c\psi_{\sigma}$ for all $\sigma\in\text{\rm Gal}(K/F)$
    and hence $T=c\cdot\text{\rm id}_V$.
\end{proof}

Put 
$$
 \Theta:W_{K/F}
        \xrightarrow{\text{\rm Ind}_{K^{\times}}^{W_{K/F}}\theta}
        GL_{\Bbb C}(V)
        \xrightarrow{\text{\rm canonical}}
        PGL_{\Bbb C}(V)
$$
and
$$
 \rho:W_F\xrightarrow{\text{\rm canonical}}W_{K/F}
         \xrightarrow{\Theta}
      PGL_{\Bbb C}(V).
$$
Then the representation of Weil-Deligne group corresponding to the 
triplet $(\rho,PGL_{\Bbb C}(V),0)$ is 
$$
 \varphi:W_F\times SL_2(\Bbb C)
             \xrightarrow{\text{\rm projection}}
         W_F\xrightarrow{\rho}PGL_{\Bbb C}(V).
$$
Let $A_{\varphi}$ be the centralizer of $\text{\rm Im}(\varphi)$ in 
$PGL_{\Bbb C}(V)$. 

Let us denote by $\mathcal{A}_{\theta}$ the set of the group homomorphism 
$\lambda:W_{K/F}\to\Bbb C^{\times}$ whose restriction to $K^{\times}$
is a character $\alpha\mapsto\theta(\alpha^{\tau-1})$ with some 
$\tau\in\text{\rm Gal}(K/F)$ which is uniquely determined by
  $\lambda$. Let us call it associated with $\lambda$. 
If $\tau\in\text{\rm Gal}(K/F)$ is associated with 
$\lambda\in\mathcal{A}_{\theta}$, then we have 
$$
 \theta(\alpha^{\sigma(\tau-1)})=\theta(\alpha^{\tau-1})
$$
for all $\alpha\in K^{\times}$ and 
$\sigma,\tau\in\text{\rm Gal}(K/F)$, because 
$$
 \lambda(\alpha^{\sigma})=\lambda((\sigma,1)^{-1}(1,\alpha)(\sigma,1))
 =\lambda(\alpha).
$$
This implies that $\mathcal{A}_{\theta}$ is in fact a subgroup of the
character group of $W_{K/F}$. 

Take a $\overline T\in A_{\varphi}$ with $T\in GL_{\Bbb C}(V)$. Then
we have a character 
$$
 \lambda:W_{K/F}\to\Bbb C^{\times}
$$
such that $gT=\lambda(g)Tg$ for all $g\in W_{K/F}$. If 
$(T\psi_1)(\tau)\neq 0$ with $\tau\in\text{\rm Gal}(K/F)$ then 
$\alpha\cdot T(\psi_1)=\lambda(\alpha)T(\alpha\cdot\psi_1)$ for
$\alpha\in K^{\times}$ implies 
$\theta(\alpha^{\tau})=\lambda(\alpha)\theta(\alpha)$
for all $\alpha\in K^{\times}$. Hence we have $T\psi_1=c\psi_{\tau}$
with $c\in\Bbb C^{\times}$. Then we have
$$
 T\psi_{\sigma}=c\cdot\lambda(\sigma)^{-1}\sigma\cdot\psi_{\tau}
 =c\cdot\lambda(\sigma)^{-1}\theta\left(\alpha_{K/F}(\sigma,\tau)\right)
  \cdot\psi_{\sigma\tau}
$$
for all $\sigma\in\text{\rm Gal}(K/F)$. We have 

\begin{prop}
\label{prop:centralizer-of-image-of-varphi-in-pgl(v)-general-case}
$\overline T\mapsto\lambda$ gives a group isomorphism of
  $A_{\varphi}$ onto $\mathcal{A}_{\theta}$.
\end{prop}
\begin{proof}
It is clear that $\overline T\mapsto\lambda$ is injective group
homomorphism, because $\text{\rm Ind}_{K^{\times}}^{W_{K/F}}\theta$ is
irreducible. 
Take any $\lambda\in\mathcal{A}_{\theta}$ and the 
$\tau\in\text{\rm Gal}(K/F)$ associated with it. 
Define a $T\in GL_{\Bbb C}(V)$ by
$$
 T\psi_{\sigma}
 =\lambda(\sigma)^{-1}\theta\left(\alpha_{K/F}(\sigma,\tau)\right)
   \cdot\psi_{\sigma\tau}
$$
for all $\sigma\in\text{\rm Gal}(K/F)$. Then we have 
$gT=\lambda(g)\cdot Tg$ for all $g\in W_{K/F}$. 
\end{proof}

From now on, we will suppose that $x\mapsto\theta(x^{\sigma-1})$ is a
trivial character of $O_K^{\times}$ only if $\sigma\in\text{\rm Gal}(K/F)$
is $1$. 

The $2$-cocycle $\alpha_{K/F}$ can be chosen
so that $\alpha_{K/F}(\sigma,\tau)\in O_K^{\times}$ for all 
$\sigma,\tau\in\text{\rm Gal}(K/K_0)$ where 
$K_0=K\cap F^{\text{\rm ur}}$ is the maximal
unramified subextension of $K/F$. Then 
the image of $I_F\subset W_F$ by the canonical surjection 
$W_F\to W_{K/F}$ is 
$\text{\rm Gal}(K/K_0)\times O_K^{\times}\subset W_{K/F}$.  Since
$K/F$ is tamely ramified, the Galois group 
$\text{\rm Gal}(K/K_0)=\langle\tau_0\rangle$ is a cyclic group of
order $e=e(K/F)$. Then $\text{\rm Gal}(K/F)$ is generated by 
$\sigma_0=\widetilde{\text{\rm Fr}}|_K$ and $\tau_0$. We have
\begin{equation}
 \sigma_0\tau_0\sigma_0^{-1}=\tau_0^m
\label{eq:relation-between-generator-of-tamely-ramified-galois-group}
\end{equation}
with some $0<m<e$ such that $\text{\rm GCD}\{e,m\}=1$. Put 
$$
 \psi_{ij}=\psi_{\tau_0^i\sigma_0^j}\in V
 \;\text{\rm with}\;\;0\leq i<e,\;\; 0\leq j<f.
$$
Then the $\Bbb C$-basis $\{\psi_{ij}\}_{i,j}$ gives the identification 
$$
 GL_{\Bbb C}(V)=GL_n(\Bbb C)
 \;\text{\rm and}\;
 PGL_{\Bbb C}(V)=PGL_n(\Bbb C).
$$
We have, for all $\alpha\in K^{\times}$ 
\begin{equation}
 \alpha\cdot\psi_{ij}
 =\theta\left(\alpha^{\tau_0^i\sigma_0^j}\right)\cdot\psi_{ij}
 \quad
 (0\leq i<e, 0\leq j<f)
\label{eq:explicit-action-of-theta-alpha}
\end{equation}
so that $\Theta(\alpha)\in PGL_n(\Bbb C)$ is diagonal. On the other
hand we have
$$
 \tau_0\cdot\psi_{ij}
 =\begin{cases}
   \theta\left(\alpha_{K/F}(\tau_0,\tau_0^i\sigma_0^j)\right)
    \cdot\psi_{i+1,j}&:0\leq i<e-1,\\
   \theta\left(\alpha_{K/F}(\tau_0,\tau_0^{-1}\sigma_0^j)\right)
    \cdot\psi_{0,j}&:i=e-1
  \end{cases}
$$
hence
\begin{equation}
 \Theta(\tau_0)=\overline{
                 \begin{bmatrix}
                  J_0&   &      &       \\
                     &J_1&      &       \\
                     &   &\ddots&       \\
                     &   &      &J_{f-1}
                 \end{bmatrix}}
 \in PGL_n(\Bbb C)
\label{eq:explicit-action-of-theta-tau-0}
\end{equation}
with 
$$
 J_j=\begin{bmatrix}
          0  &   0  &\cdots&  0      &  1   \\
          1  &   0  &      &         &  0   \\
             &   1  &\ddots&         &\vdots\\
             &      &\ddots&  0      &  0   \\
             &      &      &  1      &  0
     \end{bmatrix}
     \begin{bmatrix}
      a_{0j}&      &      &         \\
            &a_{1j}&      &         \\
            &      &\ddots&         \\
            &      &      &a_{e-1,j}
     \end{bmatrix}
$$
($a_{ij}=\theta\left(\alpha_{K/F}(\tau_0,\tau_0^i\sigma_0^j)\right)$). So
the space of the $\text{\rm Ad}\circ\varphi(I_F)$-fixed vectors in 
$\widehat{\frak g}=\frak{pgl}_n(\Bbb C)=\frak{sl}_n(\Bbb C)$ is
$$
 {\widehat{\frak{g}}}^{\text{\rm Ad}\circ\varphi(I_F)}=
 \left\{\begin{bmatrix}
          a_11_e&      &      &      \\
                &a_21_e&      &      \\
                &      &\ddots&      \\
                &      &      &a_f1_e
         \end{bmatrix}\biggm| 
          \begin{array}{l}
           a_i\in\Bbb C,\\
           a_1+a_2+\cdots+a_f=0
          \end{array}\right\}.
$$
A $\Bbb C$-basis of it is given by
\begin{equation}
 X_1=\begin{bmatrix}
      P&   &      & \\
       &0_e&      & \\
       &   &\ddots& \\
       &   &      &0_e
     \end{bmatrix},\;
 X_2=\begin{bmatrix}
       0_e& &      & \\
          &P&      & \\
          & &\ddots& \\
          & &      &0_e
      \end{bmatrix},\cdots,
  X_{f-1}=\begin{bmatrix}
           0_e&      &   &  \\
              &\ddots&   &  \\
              &      &0_e&  \\
              &      &   &P
          \end{bmatrix}
\label{eq:baisi-of-ad-varphi-i-f-fixed-subspec-of-widehat-g}
\end{equation}
with $P=\begin{bmatrix}
         1_e&   \\
            &-1_e
        \end{bmatrix}$. The relation
\eqref{eq:relation-between-generator-of-tamely-ramified-galois-group}
gives
$$
 \sigma_0\tau_0^i\sigma_0^j
 =\begin{cases}
   \tau_0^{i+m}\sigma_0^{j+1}&:0\leq j<f-1,\\
   \tau_0^{i+m}&:j=f-1.
  \end{cases}.
$$
Put $i+m\equiv i^{\prime}\npmod{e}$ for $0\leq i,i^{\prime}<e$ and let 
$[m]_e\in GL_e(\Bbb Z)$ be the permutation matrix associated with the
element
$$
 \begin{pmatrix}
  0&1&2&\cdots&e-1\\
  0^{\prime}&1^{\prime}&2^{\prime}&\cdots&(e-1)^{\prime}
 \end{pmatrix}
$$
of the symmetric group of degree $e$. Then we have
$$
 \Theta(\sigma_0)=\overline{
                   \begin{bmatrix}
                    0  & 0 &\cdots&   0   &I_{f-1}\\
                    I_0& 0 &      &       &   0   \\
                       &I_1&\ddots&       & \vdots\\
                       &   &\ddots&   0   &   0   \\
                       &   &      &I_{f-2}&   0
                   \end{bmatrix}}
$$
with
$$
 I_j=[m]_e\begin{bmatrix}
           b_{0j}&      &      &         \\
                 &b_{1j}&      &         \\
                 &      &\ddots&         \\
                 &      &      &b_{e-1,j}
          \end{bmatrix}
$$
($b_{ij}=\theta\left(\alpha_{K/F}(\sigma_0,\tau_0^i\sigma_0^j)\right)$). So
 the representation matrix of 
$\text{\rm Ad}\circ\varphi(\widetilde{\text{\rm Fr}})$ on 
${\widehat{\frak{g}}}^{\text{\rm Ad}\circ\varphi(I_F)}$ with respect
 to the basis 
\eqref{eq:baisi-of-ad-varphi-i-f-fixed-subspec-of-widehat-g} is
$$
  \begin{bmatrix}
     -1  &   1  &  0   &\cdots&  0   \\
     -1  &   0  &  1   &\cdots&  0   \\
   \vdots&\vdots&\ddots&\ddots&\vdots\\
     -1  &   0  &      &   0  &  1   \\
     -1  &   0  &\cdots&   0  &  0
  \end{bmatrix}.
$$
Hence we have
\begin{align*}
 L(\varphi,\text{\rm Ad},s)
 &=\det\left(
   1-q^{-s}\text{\rm Ad}\circ\varphi(\widetilde{\text{\rm Fr}})
    |_{\frak{g}^{\text{\rm Ad}\circ\varphi(\widetilde{\text{\rm Fr}})}}
      \right)^{-1}\\
 &=\left(1+q^{-s}+q^{-2s}+\cdots+q^{-(f-1)s}\right)^{-1}
\end{align*}
and 
\begin{equation}
 \frac{L(\varphi,\text{\rm Ad},1)}
      {L(\varphi,\text{\rm Ad},0)}
 =f\cdot\frac{1-q^{-1}}
            {1-q^{-f}}.
\label{eq:ratio-of-l-function-of-tamely-induced-rep-of-weil-group}
\end{equation}

\subsection[]{}
\label{subsec:calculation-of-artin-conductor}
In this subsection, we will prove Theorem 
\ref{th:ratio-of-gamma-factor}. To begin with, we will prove


\begin{prop}
\label{prop:rep-theta-of-relative-weil-group-is-indep-of-extension}
The group homomorphism
$$
 \Theta:W_{K/F}
        \xrightarrow{\text{\rm Ind}_{K^{\times}}^{W_{K/F}}\theta}
        GL_{\Bbb C}(V)
        \xrightarrow{\text{\rm canonical}}
        PGL_{\Bbb C}(V)
$$
is independent, up to the conjugate in 
$PGL_{\Bbb C}(V)$,  of the choice of the extension 
$\theta:K^{\times}\to\Bbb C^{\times}$ from the
character $\theta$ of $G_{\beta}(O_F/\frak{p}^r)$.
\end{prop}
\begin{proof}
Take another extension $\theta^{\prime}:K^{\times}\to\Bbb C^{\times}$
from the character $\theta$ of 
$$
 G_{\beta}(O_F/\frak{p}^r)
 =\left\{
   \overline\varepsilon\in\left(O_K/\frak{p}_K^{er}\right)^{\times}
    \biggm| N_{K/F}(\varepsilon)\equiv 1\npmod{\frak{p}^r}\right\}
$$
Then $\theta=\theta^{\prime}$ on the subgroup
$$
 U_{K/F}=\{\varepsilon\in O_K^{\times}\mid N_{K/F}(\varepsilon)=1\}
$$
of $K^{\times}$ because the canonical group homomorphism 
$U_{K/F}\to G_{\beta}(O_F/\frak{p}^r)$ is surjective. 
So there exists a character $\chi:O_F^{\times}\to\Bbb C^{\times}$
such that 
$$
 \theta^{\prime}(\varepsilon)
 =\theta(\varepsilon)\cdot\chi\left(N_{K/F}(\varepsilon)\right)
$$
for all $\varepsilon\in O_K^{\times}$. We can extend $\chi$ to a
character of $F^{\times}$ so that
$$
 \theta^{\prime}(x)=\theta(x)\cdot\chi\left(N_{K/F}(x)\right)
$$
for all $x\in K^{\times}$. The induced representations 
$\text{\rm Ind}_{K^{\times}}^{W_{K/F}}\theta$ and 
$\text{\rm Ind}_{K^{\times}}^{W_{K/F}}\theta^{\prime}$ are realized on
  the complex vector space of the complex valued functions on
  $\text{\rm Gal}(K/F)$. For any 
$\psi^{\prime}
 \in\text{\rm Ind}_{K^{\times}}^{W_{K/F}}\theta^{\prime}$, put
$\psi(\tau)=\chi(\gamma(\tau))\cdot\psi^{\prime}(\tau)$ 
($\tau\in\text{\rm Gal}(K/F)$), where
$$
 \gamma(\tau)=\prod_{\sigma\in\text{\rm Gal}(K/F)}
               \alpha_{K/F}(\tau,\sigma)
             \in F^{\times}.
$$
Note that we have
$$
 N_{K/F}(\alpha_{K/F}(\sigma,\tau))
 =\gamma(\sigma)\gamma(\sigma\tau)^{-1}\gamma(\tau)
$$
for all $\sigma,\tau\in\text{\rm Gal}(K/F)$. Then the direct
calculations show that the $\Bbb C$-linear map 
$T:\psi^{\prime}\mapsto\psi$ satisfies the relations
$$
 \sigma T=\chi(\gamma(\sigma))^{-1}\cdot T\sigma,
 \qquad
 \alpha T=\chi(N_{K/F}(\alpha))\cdot T\alpha
$$
for all $\sigma\in\text{\rm Gal}(K/F)$ and $\alpha\in K^{\times}$.
\end{proof}

We have also

\begin{prop}\label{prop:level-structure-of-theta}
Take a $\tau\in\text{\rm Gal}(K/F)$ and an integer $0\leq k<er$. Then 
$\theta(\alpha^{\tau})=\theta(\alpha)$ for all 
$\alpha\in 1+\frak{p}_K^{er-k}$ if and only if
$$
 \tau\in
 \begin{cases}
  \text{\rm Gal}(K/F)&:\text{\rm if $k<e$},\\
  \text{\rm Gal}(K/K_0)&:\text{\rm if $k=e$},\\
  \{1\} &:\text{\rm if $k>e$}.
 \end{cases}
$$
\end{prop}
\begin{proof}
We can assume that $0\leq k\leq el^{\prime}$. Then we have
$$
 (1+\varpi^r\varpi_K^{-k}x)^{\tau}(1+\varpi^r\varpi_K^{-k}x)^{-1}
 \equiv
 1+\varpi^r\left(\varpi_K^{-k\tau}x^{\tau}-\varpi_K^{-k}x\right)
 \npmod{\frak{p}_K^{er}}
$$
for all $x\in O_K$. So $\theta(\alpha^{\tau})=\theta(\alpha)$ for all 
$\alpha\in 1+\frak{p}_K^{er-k}$ if and only if
$$
 T_{K/F}(x\cdot\varpi_K^{-k}(\beta^{\tau}-\beta))\in O_F
$$
for all $x\in O_K$. This means that 
$\varpi_K^{-k}(\beta^{\tau}-\beta)
 \in\mathcal{D}(K/F)^{-1}=\frak{p}_K^{1-e}$. 
Here $\mathcal{D}(K/F)=\frak{p}_K^{e-1}$ is the different of $K/F$
because $K/F$ is tamely ramified. Since $O_K=O_F[\beta]$, the
condition is $\text{\rm ord}_K(x^{\tau}-x)\geq k-e+1$ for all 
$x\in O_K$ which means that
$$
 \tau\in\begin{cases}
         D_{-1}(K/F)=\text{\rm Gal}(K/F)&:k-e<0,\\
         D_0(K/F)=\text{\rm Gal}(K/K_0)&:k-e=0,\\
         D_1(K/F)=\{1\}                &:k-e>0
        \end{cases}
$$
where
$$
 D_t(K/F)=\{\sigma\in\text{\rm Gal}(K/F)\mid
             \text{\rm ord}_K(x^{\sigma}-x)\geq t+1\;
              \forall x\in O_K\}
$$
for $-1\leq t\in\Bbb R$ is the ramification groups of $K/F$. 
\end{proof}

In particular $\theta(\alpha^{\tau})=\theta(\alpha)$ for all 
$\alpha\in O_K^{\times}$ only if 
$\tau\in\text{\rm Gal}(K/F)$ is $1$, and hence the results of the
preceding subsection are applicable to our case. 

\begin{prop}\label{prop:a-theta-is-character-group-of-gal-k-over-f}
$\mathcal{A}_{\theta}$ is equal to the group of the character
  $\lambda$ of $W_{K/F}$ which is trivial on $K^{\times}$. In
  particular
\begin{equation}
 |A_{\varphi}|=|\mathcal{A}_{\theta}|
 =(O_K:N_{K/F}(O_K^{\times}))\cdot f.
\label{eq:order-of-centralizer-in-pgl}
\end{equation}
\end{prop}
\begin{proof}
Let $\lambda:W_{K/F}\to\Bbb C^{\times}$ be a group homomorphism such
that $\lambda(\alpha)=\theta(\alpha^{\tau-1})$ for all 
$\alpha\in K^{\times}$ with some $\tau\in\text{\rm Gal}(K/F)$. We have 
$\theta(\alpha^{\sigma(\tau-1)})=\theta(\alpha^{\tau-1})$ for all 
$\sigma\in\text{\rm Gal}(K/F)$ and $\alpha\in K^{\times}$, and hence
$$
 \theta(\alpha^{\tau-1})^n
 =\prod_{\sigma\in\text{\rm Gal}(K/F)}
   \theta(\alpha^{\sigma(\tau-1)})=1
$$
for all $\alpha\in K^{\times}$. This means that the group index 
$(O_K^{\times}:\text{\rm Ker}(\lambda|_{O_K^{\times}}))$ is finite and
a divisor of $n$, and hence prime to $p$. 
$1+\frak{p}_K^m\subset\text{\rm Ker}(\lambda|_{O_K^{\times}})$ with
some integer $m>0$. On the other hand we have 
$O_K^{\times}=\langle w\rangle\times(1+\frak{p}_K)$ 
with a primitive $q^f-1$-th root of unity $w\in K$ and 
$(1+\frak{p}_K:1+\frak{p}_K^m)=q^{f(m-1)}$ is a power of $p$. This
implies that $m=1$, that is $\theta(\alpha^{\tau})=\theta(\alpha)$ for
all $\alpha\in 1+\frak{p}_K$. Then $\tau=1$ by Proposition 
\ref{prop:level-structure-of-theta}. So $\lambda$ is trivial on
$K^{\times}$. Now we have 
$$
 |\mathcal{A}_{\theta}|=(K_1:F)=(K_1:K_0)\cdot f
$$
where $K_1$ is the maximal abelian subextension of $K/F$. On the other
hand we have $N_{K/F}(O_K^{\times})=N_{K_1/F}(O_{K_1}^{\times})$, and
hence we have
$$
 (O_F^{\times}:N_{K/F}(O_K^{\times})=e(K_1/F)=(K_1:K_0)
$$
because $K_0$ is the maximal unramified subextension of $K/F$.
\end{proof}
 
The image of $I_F\subset W_F$ under the canonical surjection
$$
 W_F\to W_F/\overline{[W_K,W_K]}=W_{K/F}
  \subset\text{\rm Gal}(K^{\text{\rm ab}}/F)
$$
is $\text{\rm Gal}(K^{\text{\rm ab}}/F^{\text{\rm ur}})$ which sits in
the group extension
$$
 1\to O_K^{\times}
  \xrightarrow{\delta_K}
  \text{\rm Gal}(K^{\text{\rm ab}}/F^{\text{\rm ur}})
  \xrightarrow{\text{\rm res.}}
  \text{\rm Gal}(K/K_0)\to 1.
$$
Let us denote by $K_k=K_{\varpi_K,k}$ ($k=1,2,\cdots$) the
field of $\varpi_K^k$-th division points of Lubin-Tate theory. 
Then we have an isomorphism 
$$
 \delta_K:1+\frak{p}_K^k\,\tilde{\to}\,
  \text{\rm Gal}(K^{\text{\rm ab}}/K_kK^{\text{\rm ur}}).
$$
Because the character $\theta:K^{\times}\to\Bbb C^{\times}$ comes from
a character of 
$$
 G_{\beta}(O_F/\frak{p}^r)\subset\left(O_K/\frak{p}_K^{er}\right)^{\times},
$$
$\Theta$ is trivial on 
$\text{\rm Gal}(K^{\text{\rm ab}}/K_{er}K^{\text{\rm ur}})$. 
Note that 
$$
 K_{er}K^{\text{\rm ur}}=K_{er}F^{\text{\rm ur}}
$$
is a finite extension of $F^{\text{\rm ur}}$. Let us use the upper
numbering 
$$
 D^s=D_t(K_{er}F^{\text{\rm ur}}/F^{\text{\rm ur}})
$$
of the higher ramification group, where $t\mapsto s$
is the inverse of Hasse function whose graph is
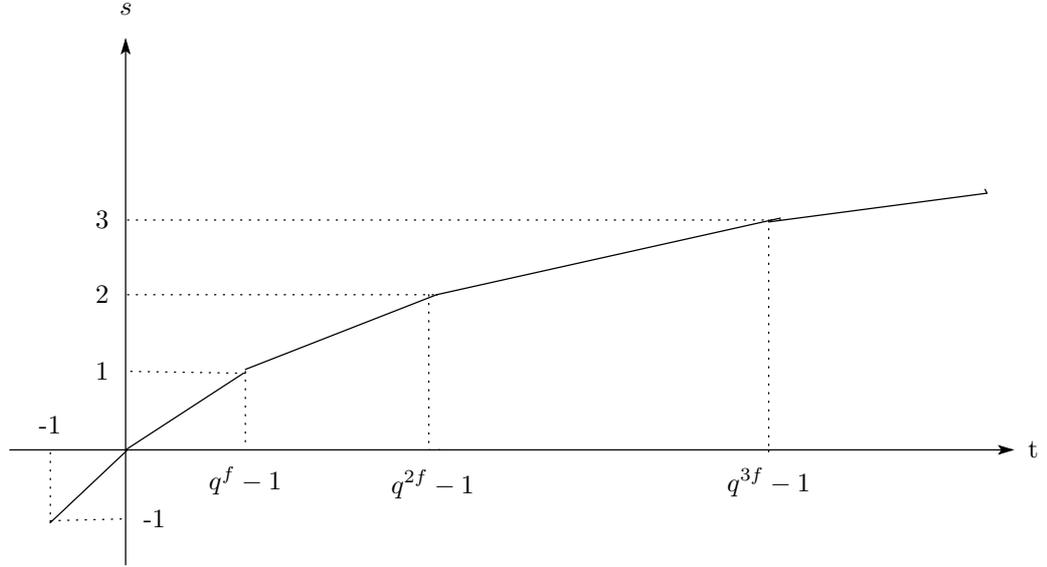
\begin{center}
\unitlength 0.1in
\begin{picture}( 52.8000, 29.6500)( 14.0000,-42.1000)
%
\special{pn 8}%
\special{pa 2000 3610}%
\special{pa 2610 3210}%
\special{fp}%
%
\special{pn 8}%
\special{pa 2620 3190}%
\special{pa 3610 2800}%
\special{fp}%
\special{pa 3610 2800}%
\special{pa 3610 2800}%
\special{fp}%
%
\special{pn 8}%
\special{pa 3610 2800}%
\special{pa 5390 2400}%
\special{fp}%
%
\special{pn 8}%
\special{pa 2010 3610}%
\special{pa 1610 3990}%
\special{fp}%
%
\special{pn 8}%
\special{pa 1610 3990}%
\special{pa 1610 3610}%
\special{dt 0.045}%
%
\special{pn 8}%
\special{pa 1610 3980}%
\special{pa 2010 3970}%
\special{dt 0.045}%
%
\special{pn 8}%
\special{pa 2620 3200}%
\special{pa 2620 3600}%
\special{dt 0.045}%
%
\special{pn 8}%
\special{pa 2620 3210}%
\special{pa 2010 3200}%
\special{dt 0.045}%
%
\special{pn 8}%
\special{pa 2010 2800}%
\special{pa 3620 2800}%
\special{dt 0.045}%
\special{pa 3620 3610}%
\special{pa 3620 3610}%
\special{dt 0.045}%
%
\special{pn 8}%
\special{pa 2010 2410}%
\special{pa 5400 2410}%
\special{dt 0.045}%
\special{pa 5330 2410}%
\special{pa 5330 3610}%
\special{dt 0.045}%
\special{pa 5330 3620}%
\special{pa 5330 3580}%
\special{dt 0.045}%
%
\special{pn 8}%
\special{pa 2000 4210}%
\special{pa 2000 1470}%
\special{fp}%
\special{sh 1}%
\special{pa 2000 1470}%
\special{pa 1980 1538}%
\special{pa 2000 1524}%
\special{pa 2020 1538}%
\special{pa 2000 1470}%
\special{fp}%
\put(18.8000,-32.0000){\makebox(0,0){1}}%
\put(18.8000,-28.0000){\makebox(0,0){2}}%
\put(18.8000,-24.1000){\makebox(0,0){3}}%
\put(21.5000,-39.7000){\makebox(0,0){-1}}%
\put(16.1000,-34.8000){\makebox(0,0){-1}}%
\put(26.2000,-37.7000){\makebox(0,0){$q^f-1$}}%
\put(35.9000,-37.9000){\makebox(0,0){$q^{2f}-1$}}%
\put(53.3000,-37.8000){\makebox(0,0){$q^{3f}-1$}}%
\put(20.0000,-13.1000){\makebox(0,0){$s$}}%
%
\special{pn 8}%
\special{pa 1400 3610}%
\special{pa 6590 3610}%
\special{fp}%
\special{sh 1}%
\special{pa 6590 3610}%
\special{pa 6524 3590}%
\special{pa 6538 3610}%
\special{pa 6524 3630}%
\special{pa 6590 3610}%
\special{fp}%
%
\special{pn 8}%
\special{pa 5330 2420}%
\special{pa 6460 2270}%
\special{fp}%
\special{pa 6460 2270}%
\special{pa 6450 2250}%
\special{fp}%
\put(67.0000,-36.1000){\makebox(0,0){t}}%
%
\special{pn 8}%
\special{pa 3570 2800}%
\special{pa 3570 3610}%
\special{dt 0.045}%
\special{pa 3570 3610}%
\special{pa 3570 3610}%
\special{dt 0.045}%
\end{picture}%

\end{center}
Then $\delta_K$ induces the isomorphism
$$
 (1+\frak{p}_K^k)/(1+\frak{p}_K^{er})\,\tilde{\to}\,
 \text{\rm Gal}(K_{er}K^{\text{\rm ur}}/K_kK^{\text{\rm ur}})
 =D^s
$$
for $k-1<s\leq k$ ($k=1,2,\cdots$), and hence
$$
 |D_t|=\begin{cases}
        e\cdot q^{nr}(1-q^{-f})&:t=0,\\
        q^{nr-fk}              &:q^{f(k-1)}-1<t\leq q^{fk}-1.
       \end{cases}
$$
The explicit actions 
\eqref{eq:explicit-action-of-theta-alpha} and 
\eqref{eq:explicit-action-of-theta-tau-0} and Proposition 
\ref{prop:level-structure-of-theta} shows that 
the space of $\text{\rm Ad}\circ\Theta(D_t)$-fixed
vectors in $\widehat{\frak g}$ is
$$
 \left\{\begin{bmatrix}
         a_11_e&      &      &      \\
               &a_21_e&      &      \\
               &      &\ddots&      \\
               &      &      &a_f1_e
        \end{bmatrix}\biggm|
           \begin{array}{l}
            a_i\in\Bbb C,\\
            a_1+a_2+\cdots+a_f=0
           \end{array}\right\}
$$
if $t=0$, 
$$
   \left\{\begin{bmatrix}
           a_1&   &      &   \\
              &a_2&      &   \\
              &   &\ddots&   \\
              &   &      &a_n
          \end{bmatrix}\biggm|
           \begin{array}{l}
            a_i\in\Bbb C,\\
            a_1+a_2+\cdots+a_n=0
           \end{array}\right\}
$$
if $0<t\leq q^{f\{e(r-1)-1\}}-1$, 
$$
    \left\{\begin{bmatrix}
           A_1&   &      &   \\
              &A_2&      &   \\
              &   &\ddots&   \\
              &   &      &A_f
          \end{bmatrix}\biggm|
           \begin{array}{l}
            A_i\in M_e(\Bbb C),\\
            \text{\rm tr}(A_1+A_2+\cdots+A_n)=0
           \end{array}\right\}
$$
if $q^{f\{e(r-1)-1\}}-1<t\leq q^{fe(r-1)}-1$ and 
$\widehat{\frak g}$ if $q^{fe(r-1)}-1<t\leq q^{fer}-1$. So we have
$$
 \dim_{\Bbb C}{\widehat{\frak g}}^{D_t}
  =\begin{cases}
   f-1   &:t=0,\\
   n-1   &:0<t\leq q^{f\{e(r-1)-1\}}-1,\\
   fe^2-1&:q^{f\{e(r-1)-1\}}-1<t\leq q^{fe(r-1)}-1,\\
   n^2-1 &:q^{fe(r-1)}-1<t\leq q^{fer}-1
  \end{cases}.
$$
Hence we have
\begin{align*}
 \sum_{t=0}^{\infty}&(D_0:D_t)^{-1}
       \dim_{\Bbb C}\left(\widehat{\frak g}/{\widehat{\frak g}}^{D_t}
                          \right)\\
 &=n^2-f+(n^2-n)\cdot\frac 1e\cdot\{e(r-1)-1\}+(n^2-fe^2)\cdot\frac 1e\\
 &=rn(n-1).
\end{align*}
Combined with
\eqref{eq:ratio-of-l-function-of-tamely-induced-rep-of-weil-group}, we
have
\begin{equation}
 \gamma(\varpi,\text{\rm Ad},0)
 =q^{rn(n-1)/2}\cdot f\cdot\frac{1-q^{-n}}
                                {1-q^{-f}}.
\label{eq:explicit-value-of-gamma-factor}
\end{equation}
The equations 
\eqref{eq:gamma-factor-of-principal-parametor}, 
\eqref{eq:order-of-centralizer-in-pgl} and
\eqref{eq:explicit-value-of-gamma-factor} prove Theorem 
\ref{th:ratio-of-gamma-factor}. 


Sendai 980-0845, Japan\\
Miyagi University of Education\\
Department of Mathematics

\begin{thebibliography}{99}

\bibitem{Cassels-Frohlich}
J.W.S.Cassels, A.Fr\"ohlich : Algebraic Number Theory 
                              (Academic Press, 1967)



\bibitem{GelbartKnapp1982}
S.S.Gerbart, A.W.Knapp : {\it $L$-indistinguishability and $R$ group
                 for the special linear group}
               (Adv. in Math. 43 (1982), 101--121)



\bibitem{GrossReeder2010}
B.H.Gross, M.Reeder : {\it Arithmetic invariants of discrete Langlands
                        parameters}
                      (Duke Math. J. 154 (2010), 431--508)

\bibitem{HiragaSaito2012}
K.Hiraga, H.Saito : On $L$-Packets for Inner Forms of $SL_n$ 
                  (Memoirs of A.M.S. 1013 (2012))

\bibitem{H-I-I2008}
K.Hiraga, A.Ichino, T.Ikeda : {\it Formal degrees and adjoint
                               $\gamma$-factors} 
   (J.Amer. Math. Soc. 21 (2008), 283--304;{\it Correction}
J.Amer. Math.Soc. 21 (2008) 1211--1213)

\bibitem{LabesseLanglands1979}
J-P.Labesse, R.P.Langlands : {\it $L$-indistinguishability for
                               $SL(2)$} 
          (Can. J. Math. 31 (1979), 726--785)

\bibitem{MoySally1984}
A.Moy, P.J.Sally, Jr. : {\it Supercuspidal representations of $SL_n$
  over a $p$-adic field: the tame case}
          (Duke Math. J. 51 (1984), 149--161)

\bibitem{Serre1971}
J.-P. Serre : {\it Cohomologies des groupes discrets} 
              (Ann. of Math. Stud. 70 (1971), 77-169)

\bibitem{Shintani1968}
T.Shintani : {\it On certain square integrable irreducible unitary
              representations of some $\frak{p}$-adic linear groups} 
            (J. Math. Soc. Japan, 20 (1968), 522--565)



\bibitem{Takase2019}
K.Takase : {\it Regular irreducible representations of classical
           reductive groups over finite quotient rings}
           (arXiv:1905.02542v1)

\end{thebibliography}
\end{document}